\documentclass[a4paper,reqno,11pt]{amsart}
\pdfoutput=1
\usepackage{amsmath,amssymb,mathrsfs,caption,subcaption,bm,wasysym}
\usepackage{fullpage}
\usepackage{stmaryrd} %for extra symbols such as \llbracket and \rrbracket
\usepackage{cite}
\usepackage{mathtools} %for dcases*
\usepackage{enumerate}
\usepackage[dvipsnames]{xcolor} % make sure this goes before tikz if using it
\usepackage[vcentermath]{youngtab}
\usepackage[boxsize=5mm]{ytableau}
\usepackage{tikz}
\usepackage{hyperref}
\hypersetup{
  colorlinks = true, %Colours links instead of ugly boxes
  linktocpage = true, %Makes the page numbers the links in the TOC instead of the text
  urlcolor = blue, %Colour for external hyperlinks
  linkcolor = blue, %Colour of internal links
  citecolor = OliveGreen, %Colour of citations
}

\allowdisplaybreaks

\numberwithin{equation}{section}
\newtheorem{lemma}{Lemma}[section]
\newtheorem{thm}[lemma]{Theorem}
\newtheorem{prop}[lemma]{Proposition}
\newtheorem{corollary}[lemma]{Corollary}

\newtheorem{thmciting}[lemma]{Theorem}
\newtheorem{propciting}[lemma]{Proposition}

\newenvironment{thmc}[2]{\begin{thmciting}{\textup{\textbf{\cite[#2]{#1}.}}}}{\end{thmciting}}

\newenvironment{propc}[2]{\begin{propciting}{\textup{\textbf{\cite[#2]{#1}.}}}}{\end{propciting}}

\theoremstyle{definition}
\newtheorem{Remark}[lemma]{Remark}
\newtheorem{Examp}[lemma]{Example}
\newtheorem*{defn}{Definition}

\theoremstyle{remark}

\newcommand\Comment[2][\relax]{\space\par\medskip\noindent%
   \fbox{\begin{minipage}{0.98\textwidth}\textbf{Comment\ifx\relax#1\else---#1\fi}\newline%
        #2\end{minipage}}\medskip
}

\def\iso{\stackrel{\sim}{\longrightarrow}}

\renewcommand{\phi}{\varphi}
\newcommand{\lam}{\lambda}
\newcommand{\la}{\lambda}

\newcommand{\Sym}{\mathfrak{S}}

\DeclareMathOperator{\Std}{\mathsf{Std}}
\DeclareMathOperator{\Arc}{\mathsf A}
\DeclareMathOperator{\DStd}{\mathsf{DStd}}
\DeclareMathOperator{\IStd}{\mathsf{IStd}}
\DeclareMathOperator{\Shape}{Shape}
\def\tab(#1){\,\mbox{\tiny$\young(#1)$}\,}
\newcommand{\reg}{\operatorname{reg}}
\newcommand{\mtt}{\mathtt}
\DeclareMathOperator{\diag}{diag}
\newcommand{\aff}{\rm aff}

\newcommand{\Pa}{\mathsf{P}}

\DeclareMathOperator{\res}{res}
\DeclareMathOperator{\Hom}{Hom}
\DeclareMathOperator{\DHom}{DHom}

\DeclareMathOperator{\Char}{char}
\DeclareMathOperator{\rad}{rad}

\newcommand{\Z}{\mathbb Z}
\newcommand{\Q}{\mathbb Q}
\newcommand{\Ga}{\Gamma}
\newcommand{\noarrow}{\:\notslash\:}

\newcommand{\Si}{\mathfrak S}
\newcommand{\cl}{\mathcal}
\newcommand{\lan}{\langle}
\newcommand{\ran}{\rangle}
\renewcommand{\O}{\mathcal O}
\newcommand{\CC}{\mathcal C}

\newcommand{\fch}{\operatorname{ch}_q}

\renewcommand{\geq}{\geqslant}
\renewcommand{\leq}{\leqslant}
\renewcommand{\ge}{\geqslant}
\renewcommand{\le}{\leqslant}

\newcommand{\dom}{\trianglerighteqslant}

\newcommand{\domby}{\trianglelefteqslant}

\newcommand{\Par}{\operatorname{\mathsf{Par}}}
\newcommand{\RPar}{\operatorname{\mathsf{RPar}}}

\newcommand{\CT}{\operatorname{\mathsf{CT}}}

\newcommand{\bi}{\text{\boldmath$i$}}
\newcommand{\bw}{\mathbf w}
\newcommand{\bj}{\text{\boldmath$j$}}

\newcommand{\de}{\delta}

\renewcommand{\k}{\kappa}
\newcommand{\La}{\Lambda}
\newcommand{\al}{\alpha}
\newcommand{\ct}{\mathtt c}
\newcommand{\be}{\beta}
\def\Y#1{\llbracket #1 \rrbracket}

\renewcommand{\tt}{\mathtt t}
\newcommand{\Tt}{\mathtt T}
\newcommand{\St}{\mathtt S}
\newcommand{\st}{\mathtt s}
\newcommand{\ut}{\mathtt u}
\newcommand{\rt}{\mathtt r}
\newcommand{\sm}{\setminus}

\newcommand{\mR}{\mathbb R}

\newcommand{\eps}{\varepsilon}
\newcommand{\C}{\mathbb C}

\newcommand{\fra}{\mathfrak}
\newcommand{\GL}{\operatorname{GL}}

\newcommand{\da}{\hspace{-1mm}\downarrow}

\begin{document}

\title[Bases of simple modules]{On bases of some simple modules of symmetric groups and Hecke algebras}

\author{Melanie de Boeck}
%\address{School of Mathematics, University of Birmingham, Edgbaston, Birmingham B15 2TT, UK}
\email{melaniedeboeck@hotmail.com}

\author{Anton Evseev}
%\address{School of Mathematics, University of Birmingham, Edgbaston, Birmingham B15 2TT, UK}
%\email{A.Evseev@bham.ac.uk}

\author{Sin\'ead Lyle}
\address{School of Mathematics, University of East Anglia, Norwich NR4 7TJ, UK}
\email{s.lyle@uea.ac.uk}

\author{Liron Speyer}
\address{Osaka University, Suita, Osaka 565-0871, Japan}
\email{l.speyer@ist.osaka-u.ac.jp}
%\keywords{Path, Regular, Tableau, Specht module, Decomposition number}
\thanks{The first and second authors were supported by the EPSRC grant EP/L027283/1. The fourth author was supported by an LMS Postdoctoral Mobility Grant and a Japan Society for the Promotion of Science Fellowship.}
\dedicatory{We record with deep sadness the passing of Anton Evseev on 21st February 2017.}
%%%%%%%%%% ABSTRACT %%%%%%%%%%

\begin{abstract}
We consider simple modules for a Hecke algebra with a parameter of quantum characteristic $e$.
Equivalently, we consider simple modules $D^\la$, labelled by $e$-restricted partitions $\la$ of $n$, for
 a cyclotomic KLR algebra $R_n^{\La_0}$ %of type $A_{e-1}^{(1)}$
 over a field of characteristic $p\ge 0$, with mild restrictions on $p$.
If all parts of $\la$ are at most $2$, we identify a set 
$\DStd_{e,p}(\la)$ of standard $\la$-tableaux, which is defined combinatorially and naturally labels a basis of $D^\la$. In particular, we prove that the $q$-character of $D^\la$ can be described in terms of $\DStd_{e,p}(\la)$.
We show that a certain natural approach to constructing a basis of an arbitrary $D^\la$  does not work in general, giving a counterexample to a conjecture of Mathas.
\end{abstract}

\maketitle
%\tableofcontents

%%%%%%%%%% INTRODUCTION %%%%%%%%%%

\section{Introduction}\label{SIntro}
Let $K$ be a field with a Hecke parameter $0\ne \xi \in K$ of quantum characteristic $e\in \Z_{\ge 2}$. 
We consider the Iwahori--Hecke $K$-algebra $\cl H_n (\xi)$. An important special case occurs when $\xi=1$ and $K$ has characteristic $e$, which implies that $\cl H_n (\xi) = K \Sym_n$ is the group algebra of a symmetric group. 

The Specht $\cl H_n (\xi)$-modules 
$S^\la_{\cl H}$, parameterised by partitions $\la$ of $n$, play an important role in the representation theory of $\cl H_n (\xi)$. 
In particular, if $\la$ is an $e$-restricted partition, then 
$S^\la_{\cl H}$ has a simple head $D^\la$, and all simple $\cl H_n (\xi)$-modules occur in this way. The Specht module $S^\la_{\cl H}$ has a Murphy basis indexed by the set $\Std(\la)$ of all standard $\la$-tableaux. In this paper, we investigate whether there is a subset of 
$\Std(\la)$ that 
naturally labels a basis of $D^\la$. 

This question can be made much more precise via the language of Khovanov--Lauda--Rouquier (KLR) algebras~\cite{KL09, R08}, which is used throughout the paper. 
Given an arbitrary commutative ring $\O$, we consider the cyclotomic KLR $\O$-algebra $R^{\La_0}_{n,\O}$ of type $A_{e-1}^{(1)}$, which has a natural $\Z$-grading, see~\S\ref{SSKLR}.
Brundan and Kleshchev~\cite{bkisom} and Rouquier~\cite{R08} proved that $R^{\La_0}_{n,K}$ is isomorphic to $\cl H_n (\xi)$. 
Further, Kleshchev, Mathas and Ram~\cite{KMR12} constructed a universal Specht $R^{\La_0}_{n,\O}$-module $S^\la_\O$ by explicit generators and relations such that, in particular, $S^\la_K$ is isomorphic to the $\cl H_n (\xi)$-module $S^\la_{\cl H}$. 
We denote by $D^\la_K$ the (simple) head of $S^\la_K$ 
if the partition $\la$ is $e$-restricted and set $D^\la_K:=0$ otherwise. 

The algebra $R^{\La_0}_{n,\O}$ is equipped with an orthogonal family of idempotents $\{1_{\bi} \mid \bi \in I^n\}$, where $I:= \Z/e\Z$. 
The {\em $q$-character} of a finite-dimensional $R^{\La_0}_{n,\O}$-module $M$ is defined by 
\begin{equation}\label{Ech1}
\fch M:= \sum_{\bi \in I^n} \dim_q (1_{\bi} M) \cdot \bi \in 
\langle I^n\rangle, 
\end{equation}
where $\langle I^n\rangle$ is the free $\Z[q,q^{-1}]$-module with basis $I^n$ and $\dim_q (1_{\bi} M) \in \Z[q,q^{-1}]$ is the graded dimension of $1_{\bi} M$, see~\S\ref{SSGraded}.

Let $\la$ be a partition  of $n$.
To each standard tableau $\tt\in\Std(\la)$ one attaches its {\em residue sequence} $\bi^\tt \in I^n$ and {\em degree} $\deg(\tt)\in \Z$, which are both defined combinatorially, see~\cite{BKW11} or~\S\ref{SSPartitions}. 
Then the Specht module $S^\la_\O$ has an $\O$-basis 
$\{ v^\tt \mid \tt\in \Std(\la) \}$ such that 
$1_{\bi} v^\tt = \delta_{\bi,\bi^\tt} v^\tt$ for any $\bi \in I^n$ and 
$v^\tt$ is homogeneous of degree $\deg(\tt)$ for each $\tt$. 
In particular, defining the {\em $q$-character}
of any finite set $\cl T$ of standard tableaux by 
\begin{equation}\label{Ech2}
\fch \cl T:= \sum_{\tt \in \cl T} q^{\deg (\tt)} \cdot \bi^\tt,
\end{equation}
we have
\begin{equation}\label{ESpechtChar}
 \fch S^\la_K = \fch \Std(\la).
\end{equation}
Therefore, it is reasonable to require that a desired 
subset of $\Std(\la)$ corresponding to a basis of $D^\la_K$ should have 
$q$-character equal to $\fch D^\la_K$. 
Our main results give a combinatorial construction of such a subset 
of $\Std(\la)$ for an arbitrary field $K$ (as above) 
when $\la=(\la_1,\dots,\la_l)$ satisfies $\la_1\le 2$; we refer to such partitions $\la$ as {\em 2-column partitions}. 
We refer the reader to~\cite[\S 3.3]{HM15} for a further discussion of the problem in general. 

In \S\ref{SSTwo}, we give a combinatorial definition of a subset
 $\DStd_e (\la)$ of $\Std(\la)$ for every 2-column partition $\la$. 
In order to describe $\DStd_e (\la)$, we represent standard tableaux as paths in a weight space of Dynkin type $A_1$ and construct a {\em regularisation map} $\reg_e$ on standard tableaux, which plays a key role throughout. 
 
The following theorem shows that, when $\Char K=0$, the set 
$\DStd_e(\la)$ 
labels a basis of $D^\la_K$ and, moreover,
 the composition series of 
$S^\la_K$ can be lifted to an arbitrary commutative ring 
 $\O$ in an explicit way. 
 
 \begin{thm}\label{TMain1}
 Let $\la$ be a 2-column partition of $n$. 
 \begin{enumerate}[(i)]
 \item The $\O$-span $U^\la_\O$ of $\{ v^\tt \mid \tt\in \Std(\la)\sm \DStd_e (\la) \}$ is an $R^{\La_0}_{n,\O}$-submodule of $S^\la_\O$. 
 \item 
If $\Char K=0$ then 
 there is an isomorphism $S^\la_K/U^\la_K \cong D^\la_K$
 of graded $R^{\La_0}_{n,K}$-modules. In particular, 
 $\fch D^\la_K = \fch \DStd_e (\la)$. 
 \item 
 Set $\tilde D^\la_\O:=S^\la_\O/U^\la_\O$. Let $\la=(2^x,1^y)$ 
 with 
 $y \equiv -j-1 \pmod e$ for some $0\le j<e$. 
 There is 
 an isomorphism of graded $R^{\La_0}_{n,\O}$-modules
 \[
 U^\la_\O \cong 
 \begin{dcases*}
  \tilde D^\mu_\O \langle 1 \rangle & if $j\neq 0$ and $x\ge j$; \\
  0 & otherwise,
 \end{dcases*}
 \]
 where $\mu=(2^{x-j}, 1^{y+2j})$ and 
 $\tilde D^\mu_\O \langle 1 \rangle$ denotes $\tilde D^\mu_\O$ with the grading shifted by $1$.
 \end{enumerate}
 \end{thm}

Remarkably, the aforementioned construction of $\DStd_e(\la)$ also leads to a combinatorial description of the $q$-character of $D^\la_K$ when $K$ has positive characteristic. Indeed, given a prime $p$ and a 2-column partition $\la$, define 
\[
\DStd_{e,p}(\la):= \bigcap_{z\in \Z_{\ge 0}} \DStd_{ep^z} (\la).
\]

\begin{thm}\label{TMainChar}
If $\Char K=p>0$ then $\fch D^\la_K = \fch \DStd_{e,p}(\la)$ for every 2-column partition $\la$. 
\end{thm}

Now suppose that $\Char K= p\ge 0$ and let $\la$ be a 2-column partition. 
James~\cite{J78,J84} and Donkin~\cite{D98}
determined the ungraded composition multiplicities of $S^\la_K$. In particular, 
each $D^\mu_K$ appears as a composition factor of $S^\la_K$ with multiplicity at most $1$.
Moreover, in the case when $p=e$ (i.e.~that of symmetric groups), Erdmann~\cite{Erdmann} has given a more direct description of the dimensions of simple modules labelled by $e$-restricted $2$-column partitions: she proved that these dimensions are coefficients in certain explicitly determined generating functions. 

Theorem~\ref{TDec1} extends the results of James and Donkin to give {\em graded} decomposition numbers. Thus, whenever $D^\mu_K$ appears as an (ungraded) composition factor of $S^\la_K$, there is an explicitly described integer $r_{e,p,\la,\mu} \in \{0,1\}$ such that 
$D^\mu_K \langle r_{e,p,\la,\mu}\rangle$ is a graded composition factor of $S^\la_K$. Combining this fact with Theorems~\ref{TMain1}(ii) and~\ref{TMainChar}, we obtain the character identity
\begin{equation}\label{EChSum1}
 \fch S^\la = \sum_\mu q^{r_{e,p,\la,\mu}} \fch \DStd_{e,p} (\mu),
\end{equation}
where the sum is over $2$-column partitions $\mu$ such that 
$D^\mu_K$ is a composition factor of $S^\la_K$,
and where we set $\DStd_{e,0} (\mu):= \DStd_e (\mu)$. 
In Section~\ref{SComb}, for any 2-column partitions $\la,\mu$, 
we identify an explicit subset 
 $\Std_{e,p,\mu} (\la)$ of $\Std(\la)$, which may be seen to correspond to the composition factors $D^\mu_K$ in $S^\la_K$. More precisely, $\Std_{e,p,\mu} (\la)\ne \varnothing$ if and only if
$D^\mu_K$ is a composition factor of $S^\la_K$, and if this is the case, then
\begin{equation}\label{EChSum2}
 \fch \Std_{e,p,\mu} (\la) = q^{r_{e,p,\la,\mu}} \fch \DStd_{e,p} (\mu),
\end{equation}
see Theorem~\ref{Treplamu}. The identity~\eqref{EChSum2} is proved via an explicit bijection 
\[
\reg'_{e,p,\la,\mu} \colon \Std_{e,p,\mu} (\la) \iso \DStd_{e,p} (\mu).
\]
Furthermore, there is a decomposition 
\[
 \Std (\la) = \bigsqcup_{\mu} \Std_{e,p,\mu} (\la),
\]
which may be viewed as a combinatorial lifting of the identity~\eqref{EChSum1}. 

The sets $\Std_{e,p,\mu} (\la)$ are defined in terms of a map 
$\reg_{e,p}$ from the set of 2-column standard tableaux to itself, which generalises the aforementioned regularisation map $\reg_e$. 
In fact, graded decomposition numbers 
for 2-column partitions can also be described in terms of 
$\reg_{e,p}$, see
Theorem~\ref{conj:83}.
The simple module $D^\la_K$ is self-dual, which implies that 
$\fch D^\la_K= \fch \DStd_{e,p} (\la)$ is invariant under the involution given by $q\mapsto q^{-1}$. A combinatorial proof of this fact is given in Remark~\ref{RInv}. \smallskip 

The paper is organised as follows. 
In Section~\ref{SPrelim}, we review cyclotomic KLR algebras, their Specht modules and the connection with representations of Hecke algebras. 
In Section~\ref{SPaths},
we associate a path in a weight space of type $A_{k-1}$ with every standard tableau whose shape is a $k$-column partition (for 
$k\in \Z_{\ge 2}$) and describe the degrees of standard tableaux in the language of paths.  We define 
the aforementioned regularisation map $\reg_e$ 
on $\Std(\la)$ and the set $\DStd_e (\la)$ when $\la$ has at most $2$ columns.
 
Section~\ref{SComb} is combinatorial: we prove 
Theorem~\ref{TMainChar} and the results outlined after the statement of 
that theorem. The order in which the results are proved is different from the one above. In particular, Theorem~\ref{TMainChar} is obtained as a consequence of the identities~\eqref{EChSum1} and~\eqref{EChSum2}. 

In Section~\ref{SHom}, we consider homomorphisms between 2-column Specht modules. Using a row removal result from~\cite{fs16}, we construct a homomorphism from $S^\mu_\O$ to $S^\la_\O$, where 
$\la$ and $\mu$ are as in Theorem~\ref{TMain1}(iii), and we describe explicitly the kernel and image of this homomorphism, 
see Theorems~\ref{Tphilamu} and~\ref{TKerIm}.
This leads to a proof of Theorem~\ref{TMain1}. 
 We also construct exact sequences of homomorphisms between 2-column Specht modules, 
see Corollary~\ref{genexact}. 

Finally, in Section~\ref{SCounter}, we remove the condition that $\la_1\le 2$ and consider a natural approach to extending the definition of the set $\DStd_{e} (\la)$ to an arbitrary partition $\lambda$ of $n$,  
based on the structure of $S^\la_\Q$ and its radical $\rad S^\la_\Q$, in the spirit of~\cite[\S 3.3]{HM15}. 
 We give an example showing that in some cases the resulting set $\DStd_e (\la)$ is `too big', which yields a counterexample to a conjecture of Mathas~\cite{M15}. 

Throughout, given $a,b\in \Z$, we write
$
[a,b]:= \{ c\in \Z \mid a\le c \le b\}.
$
If $b\ge 0$, we often abbreviate $a,\dots,a$ (with $b$ entries) as $a^b$. 
If $X$ is a collection of elements of an $\O$-module, 
we denote the $\O$-span of $X$ by 
$\langle X \rangle_{\O}$. The $\Z$-rank of a free $\Z$-module $U$ of finite rank is denoted by $\dim_\Z U$. 
If $1\le r<n$ are integers, we set $s_r:=(r,r+1)$ to be the corresponding elementary transposition in the symmetric group $\Sym_n$.

%%%%%%%% BASICS ON KLR ALGEBRAS %%%%%%%%%%%%%%%

\section{KLR algebras and Specht modules}\label{SPrelim}

We fix an integer $e\ge 2$ throughout the paper.  We set 
$I=\Z/e\Z = \{0,1,\ldots,e-1\}$, abbreviating $i+e\Z$ as $i$ 
(for $0\le i<e$) when there is no possibility of confusion.
For any $n\in \Z_{\ge 0}$, we
write $I^n = I\times \cdots \times I$. 
We define $\langle I^n\rangle$ to be the free $\Z[q,q^{-1}]$-module with basis $I^n$.
The symmetric group $\Si_n$ acts on the left on $I^n$ by place permutations.
An element of $I^n$ denoted by
$\bi$ is assumed to be equal to $(i_1,\ldots,i_n)$; we adopt a similar convention for other bold symbols.

\subsection{Graded algebras and modules}\label{SSGraded}
By a graded module (over any ring) we mean a $\Z$-graded one. If $V$ is a graded module and $m\in \Z$, we denote the $m$-th graded component of $V$ by $V_m$. The {\em graded dimension} of a finite-dimensional graded vector space $V$ is 
$\dim_q V:= \sum_{m\in \Z} (\dim V_m) q^m \in \Z[q,q^{-1}]$. 

Let $A$ be a graded $\O$-algebra, where $\O$ is a commutative ring. If $M=\bigoplus_{m\in \Z} M_m$ is a graded $A$-module then, for any $k\in \Z$, we write $M\langle k\rangle$ to denote the \emph{graded shift} of $M$ by $k$, which has the same structure as $M$ as an $A$-module and grading given by 
$M\langle k\rangle_m = M_{m-k}$ for all $m\in \Z$. 
If $M$ and $N$ are graded $A$-modules, then $\Hom_A (M,N)$ denotes the $\O$-module of $A$-homomorphisms from $M$ to $N$ as ungraded modules. Moreover, if $M$ is finitely generated as an $A$-module, then 
$\Hom_A (M,N)$ is graded by the following rule:
given $\phi\in \Hom_A (M,N)$ and $m\in \Z$, $\phi\in \Hom_A(M,N)_m$ if and only if $\phi(M_k)\subseteq N_{k+m}$ for all $k\in \Z$.
If $\O$ is a field, then by a {\em composition factor} of a finite-dimensional $A$-module $M$ we mean a composition factor of $M$ as an ungraded module, unless we explicitly specify otherwise. 

For every $f=f(q)\in \Z[q,q^{-1}]$, we write $\bar f(q):=f(q^{-1})$. 
This yields an involution 
\begin{equation}\label{EInvChar}
\bar{\phantom{C}} \colon \langle I^n \rangle \to \langle I^n \rangle, \;
\sum_{\bi\in I^n} f_\bi \cdot \bi \mapsto
 \sum_{\bi\in I^n} \bar f_{\bi} \cdot \bi. 
 \end{equation}

\subsection{KLR algebras}\label{SSKLR}

Consider the quiver $\Ga$ that has vertex set $I$, an arrow 
$i\leftarrow i+1$ for each $i\in I$ and no other arrows. We write $i\to j$ and $j\leftarrow i$ if there is an arrow from $i$ to $j$ but not from $j$ to $i$, and we write $i\leftrightarrows j$ if there are arrows between $i$ and $j$ in both directions (which only happens for $e=2$). 
Further, we write $i\noarrow j$ 
if $i\ne j$ and there is no arrow between $i$ and $j$ in either direction.
The quiver $\Ga$ corresponds to the Cartan matrix $C=(\ct_{ij})_{i,j\in I}$ of the affine type $A_{e-1}^{(1)}$, given by 
\[\ct_{ij}=
\begin{dcases*}
2 & if $i=j$; \\
-1 & if $i\to j$ or $j \to i$; \\
-2 & if $i \leftrightarrows j$; \\
0 & if $i \noarrow j$. 
\end{dcases*}\]

Let $\O$ be a commutative ring and let $n\in \Z_{\ge 0}$.
The {\em KLR algebra} 
$R_{n}= R_{n,\O}$ is the $\O$-algebra generated by the elements 
\[
\{ 1_{\bi} \mid \bi \in I^n \} 
\cup 
\{ \psi_r \mid 1\le r<n \}
\cup
\{ y_r \mid 1\le r\le n\}
\]
subject only to the following relations:
\begin{align}
1_{\bi} 1_{\bj} &= \de_{\bi,\bj} 1_{\bi}, \quad 
\sum_{\bi\in I^n} 1_{\bi} =1, \\
 y_{r} 1_\bi &= 1_{\bi} y_r, \\ 
\psi_r 1_{\bi} &= 1_{s_r\bi} \psi_r, \\
 y_r y_s &= y_s y_r, \label{rel:ypsi1} \\
\psi_r y_{r+1} 1_{\bi} &= (y_r \psi_r+\de_{i_r,i_{r+1}}) 1_{\bi}, 
\label{rel:ypsi2} \\ 
y_{r+1} \psi_r 1_{\bi} &= (\psi_r y_r +\de_{i_r,i_{r+1}}) 1_{\bi}, 
\label{rel:ypsi3} \\
\label{rel:ypsi4}
\psi_r y_s &=\mathrlap{ y_s \psi_r}\hphantom{\smash{\begin{cases} (y_{r+1} - y_r)(y_r-y_{r+1})1_{\bi} \\\\\\\end{cases}}} \kern-\nulldelimiterspace \text{if } s\ne r,r+1, \\
\psi_r \psi_s &=\mathrlap{\psi_s \psi_r}
\hphantom{\smash{\begin{cases} (y_{r+1} - y_r)(y_r-y_{r+1})1_{\bi} \\\\\\\end{cases}}} \kern-\nulldelimiterspace \text{if } |r-s|>1, \label{rel:commpsi} \\
\psi_r^2 1_{\bi} &= 
\begin{cases}
 0 & \text{if } i_r = i_{r+1}; \\
 1_{\bi} & \text{if } i_r \noarrow i_{r+1}; \\
 (y_{r+1} - y_r)1_{\bi} & \text{if } i_r \to i_{r+1}; \\
 (y_r - y_{r+1}) 1_{\bi} & \text{if } i_r \leftarrow i_{r+1}; \\
  (y_{r+1} - y_r)(y_r-y_{r+1})1_{\bi} & 
 \text{if } i_r \leftrightarrows i_{r+1},
\end{cases}  \label{rel:quad} \\
(\psi_r \psi_{r+1} \psi_r - \psi_{r+1}\psi_r \psi_{r+1})1_\bi &=
\begin{cases}
\mathrlap{1_{\bi}}\hphantom{\smash{(y_{r+1} - y_r)(y_r-y_{r+1})1_{\bi}}}& \text{if } i_{r+2}=i_{r} \to i_{r+1}; \\
-1_{\bi} & \text{if } i_{r+2} = i_r \leftarrow i_{r+1}; \\
(- 2 y_{r+1} + y_r + y_{r+2})1_{\bi}
& \text{if } i_{r+2} =i_r \leftrightarrows i_{r+1}; \\
0 & \text{otherwise}
\end{cases} \label{rel:braid}
\end{align}
for all $\bi\in I^n$ and all admissible $r,s$ (see~\cite{KL09, R08}).

Consider a root system with Cartan matrix $C$, with 
simple coroots $\{ \be^{\vee}_0,\ldots,\be^{\vee}_{e-1} \}$, see~\cite{K90}. 
To each fundamental dominant weight $\La$ of this root system, 
one attaches a {\em cyclotomic} quotient $R^{\La}_{n}$ of $R_n$. 
In this paper, we will only consider the {\em cyclotomic KLR algebra}
$R^{\La_0}_n$, where $\La_0$ 
is a (level $1$) weight satisfying 
$\lan \La_0, \be_i^{\vee} \ran = \de_{i,0}$ for all $i\in I$. 
The algebra $R^{\La_0}_n=R^{\La_0}_{n,\O}$ is
defined as the quotient of $R_n$ by the 2-sided ideal that is generated by the set 
\[\{1_\bi\mid\bi\in I^n,~i_1\neq 0\}\cup\{y_1\}.\]
The algebras $R_n$ and $R^{\La_0}_n$ are both $\Z$-graded with 
\[\deg(1_\bi)=0,\quad\deg(y_r)=2,\quad\deg(\psi_r1_\bi)=-\ct_{i_r i_{r+1}}\]
for all $\bi\in I^n$ and all admissible $r$. 

We fix a {\em reduced expression} for every $w\in \Si_n$, 
i.e.~a decomposition $w=s_{r_1} \dots s_{r_m}$ as a product of elementary transpositions with $m$ as small as possible. Define 
\begin{equation}\label{Epsiw}
\psi_w: = \psi_{r_1} \dots \psi_{r_m} \in R_n,
\end{equation}
noting that (in general) $\psi_w$ 
depends on the choice of a reduced expression for $w$. By definition, the {\em length} of $w$ is $\ell(w):=m$. 

\subsection{Partitions, tableaux and Specht modules}\label{SSPartitions}

A \emph{partition} is a non-increasing sequence $\la=(\la_1,\ldots,\la_l)$ of positive integers. 
As usual, we write  $|\la| = \sum_j \la_j$ and say that $\la$ is a partition of $n:=|\la|$. The unique partition of $0$ will be denoted by $\varnothing$. We always set $\la_r: =0$ for all $r>l$. 
We say that $\la$ is {\em $e$-restricted} if $\la_r - \la_{r+1} <e$ for all $r\in \Z_{>0}$.
We denote the set of partitions of $n$  by $\Par(n)$ and the set of $e$-restricted partitions of $n$ by $\RPar_e (n)$. 
Given $k\in \Z_{>0}$, 
define 
\[
\Par_{\le k}(n):= \{\la \in \Par (n) \mid \la_1 \le k \} \quad \text{and}
\quad \RPar_{e,\le k}(n):= \RPar_e (n) \cap \Par_{\le k} (n). 
\]
The {\em dominance} partial order $\domby$ on $\Par(n)$ is defined as follows: for any $\la,\mu\in \Par(n)$, we set $\mu \domby \la$ if 
$\sum_{j=1}^r \mu_j \le \sum_{j=1}^r \la_j$ for all $r\in \Z_{>0}$. 

The Young diagram of $\la$ is the subset 
$\Y\la = \{ (a,b) \mid 1\le a\le l,~ 
1\le b\le \la_a\}$ of $\Z_{>0} \times \Z_{>0}$.  
When drawing diagrams, we represent a node $(a,b)$ as the intersection of row $a$ and column $b$, with the rows numbered from the top down and the columns from left to right.

A {\em standard tableau} of size $n\in \Z_{\ge 0}$ 
is an injective map $\mtt t\colon \{1,\dots,n\} \to \Z_{>0} \times \Z_{>0}$ such that
\begin{enumerate}[(i)]
\item 
the image of $\mtt t$ is the Young diagram of some partition $\lambda$ of $n$; and  
\item 
the entries of $\mtt t$ are increasing along rows and down columns, 
i.e.~whenever $(a,b), (c,d) \in \Y\la$ are such that $a\le c$ and $b\le d$, we have $\tt^{-1} (a,b) \le \tt^{-1} (c,d)$. 
\end{enumerate}
In this situation, we refer to $\tt$ as a standard tableau of {\em shape} 
$\lambda$ and write  $\la=\Shape(\tt)$. 
If $0\le m\le n$, we denote by $\tt \da_{m}$ the restriction of $\tt$ to 
$\{1,\dots,m\}$. 
The set of all standard tableaux of shape $\la$ is denoted by $\Std(\la)$. 
For any $k\in \Z_{>0}$, we set 
$\Std_{\le k} (n):= \bigcup_{\la\in \Par_{\le k}(n)} \Std(\la)$.
If $\tt,\st\in \Std(\la)$, we write $\tt \dom \st$ and say that $\tt$ dominates $\st$ if 
$\Shape(\tt \!\da_m) \dom \Shape (\st \!\da_m)$ for all $0\le m\le n$. 
 We define $\mtt t^{\la}$ as the standard $\la$-tableau obtained by filling each row successively, going from the top down, so that 
 $\mtt t^{\la} ( \la_1+\cdots+\la_{a-1}+b) = (a,b)$ 
 for all $(a,b) \in \Y\la$.
Similarly, $\tt_\la\in \Std(\la)$ is obtained by successively filling each column, going from left to right, so  
$\tt_\la  (\la'_1+\dots+\la'_{b-1} + a) = (a,b)$,
where $\la'_j:=\# \{ r \in [1,l] \mid \la_r \ge j \}$ 
for all $j\in \Z_{>0}$. 

The symmetric group $\Si_n$ acts on the set of all bijections 
$\tt\colon \{1,\dots,n\}\to \Y\la$ as follows: 
$(g \tt)(r) =  \tt(g^{-1} r)$ for all $g\in \Si_n$ and $1\le r\le n$. 
For every $\tt\in \Std(\la)$, let $d(\tt)\in \Si_n$ be the unique element such that $d(\tt)\tt^\la= \tt$. 

By a {\em column tableau} of size $n$ we mean an injective map 
$\tt\colon \{1,\dots,n\} \to \Z_{>0} \times \Z_{>0}$ such that, whenever $(a,b) \in \tt(\{1,\dots,n\})$ and $a>1$, we have $(a-1,b)\in \tt(\{1,\dots,n\})$ and $\tt^{-1} (a-1,b) < \tt^{-1} (a,b)$. (That is, in particular, $\tt$ is required to increase down columns.)
For any $k\in \Z_{>0}$, we denote by $\CT_{\le k} (n)$ the set of column tableaux $\tt$ of size $n$ such that the image of $\tt$ is contained in $\Z_{>0} \times \{1,\dots,k\}$ (i.e.~the entries of $\tt$ all belong to the first $k$ columns). 
Note that $\Std_{\le k} (n) \subseteq \CT_{\le k} (n)$.

The \emph{residue} of a node $(a,b)\in\Z_{>0}\times\Z_{>0}$ is defined as $\res(a,b) = b-a + e\Z \in I$. We refer to a node of residue $i$ as an $i$-\emph{node}.
 The \emph{residue sequence} of a column tableau $\mtt t$ 
 is \[
\bi^\tt:=\big(\!\res\!\big(\tt(1)\big),\dots,\res\!\big(\tt(n)\big)\big)\in I^n.\]
The set of all standard $\la$-tableaux with a given residue sequence 
$\bi\in I^n$ is denoted by $\Std(\la,\bi)$. 
 
For each standard tableau $\tt\in \Std(\la)$, the \emph{degree} $\deg_e (\mtt t)$ of $\tt$ is defined in~\cite{BKW11} as follows. 
A node $(a,b)\in \Z_{>0}$ is said to be \emph{addable} for $\la$ if $(a,b)\notin \la$ and $\Y\la \cup \{(a,b)\}$ is the Young diagram of a partition. We say that $(a,b)$ is a \emph{removable} node of $\la$ if $(a,b)\in \Y\la$ and 
$\Y\la\sm \{(a,b)\}$ is the Young diagram of a partition. 
A node $(a,b)$ is said to be below a node $(a',b')$ if $a>a'$.  
If $(a,b)$ is a removable $i$-node of $\la$, define 
\[d_{(a,b)}(\la):=\#\{\text{addable}~i\text{-nodes for}~\la~\text{below}~(a,b)\}-\#\{\text{removable}~i\text{-nodes of}~\la~\text{below}~(a,b)\}.\]
Finally, we define the degree of the unique $\varnothing$-tableau to be $0$ and define recursively
\[
\deg_e (\tt) := d_{\tt(n)} (\la) + \deg_e(\tt \da_{n-1})
\]
for $\tt \in \Std(\la)$.

If $\tt \in \Std(\la)$ and $1\le r\le s\le n$, we write $r\to_{\tt} s$ if 
$\tt (r)$ and $\tt (s)$ are in the same row of $\Y\la$. 
We also write $\bi^\la := \bi^{\tt^\la}$.  

Let $\O$ be a commutative ring.  
We refer the reader to~\cite[Section~5]{KMR12} for the definition of a {\em Garnir node} $A\in \Y\la$ and the corresponding {\em Garnir element} $g^A \in R_n=R_{n,\O}$.
The {\em universal row Specht module} $S^{\la}=S^\la_{\O}$ 
is defined in~\cite{KMR12} as the left $R_n$-module 
generated by a single generator $v^{\la}$ subject only to the relations
\begin{alignat}{2}
1_\bi v^\la &= \de_{\bi,\bi^\la}v^\la, &&  \label{rel:Sp1}\\
\psi_r v^\la &= 0 &\quad&\text{if}~r\to_{\tt^\la}r+1, \label{rel:Sp2}\\
y_r v^\la &= 0, && \label{rel:Sp3} \\
g^A v^\la &= 0 &\quad& \text{for all Garnir nodes}~A~\text{of}~\Y\la,
\end{alignat}
for all $\bi\in I^n$ and all admissible $r\in\{1,\ldots,n\}$.
By~\cite[Corollary 6.26]{KMR12}, the action of $R_n$ on $S^{\la}$ factors through $R^{\La_0}_n$, so $S^{\la}$ is naturally an $R^{\La_0}_n$-module. 
For each $\tt\in \Std(\la)$,  we set
\[v^\tt:=\psi_{d(\tt)} v^\la,\] 
noting that in general $v^{\tt}$ depends on the choice of the reduced expression for $d(\tt)$ made in~\eqref{Epsiw}. 
In particular, $v^{\tt^\la}=v^\la$.

\begin{propc}{KMR12}{Proposition~5.14 and Corollary~6.24}
\label{Pbasis}
Let $\la\in \Par(n)$. 
The Specht module $S^{\la}$ is free as an $\O$-module, with basis
$\{ v^{\tt} \mid \tt \in \Std(\la)\}$. 
Moreover, $S^{\la}$ is a graded $R^{\La_0}_n$-module, with each $v^{\tt}$  homogeneous of degree $\deg_e(\tt)$. 
\end{propc}

\begin{corollary}\label{CScalarExt}
Let $\la\in \Par(n)$.
The graded $R^{\La_0}_{n,\Z}$-module $S^\la_\Z$ is isomorphic to the 
$\Z$-span of $\{ v^{\tt} \otimes 1 \mid \tt\in \Std(\la)\}$ in the 
$R^{\La_0}_{n,\Z}$-module 
$S^\la_\Z\otimes_\Z \mathbb \Q$. 
\end{corollary}

\subsection{Hecke algebras at roots of unity, a cellular basis and simple modules}\label{SSHecke}
Let $F$ be a field such that---setting $p:=\Char F$---we have that $p=0$, $p=e$ or $p$ is coprime to $e$.
Let the field $K$ be an extension of $F$, and assume that $\xi \in K \sm \{0\}$
has {\em quantum characteristic} $e$, 
i.e.~$e$ is the smallest positive integer such that $1+\xi+\cdots+\xi^{e-1}=0$.

The Iwahori--Hecke algebra $\cl H_{n} (\xi)$ is the $K$-algebra generated by
\[
 T_1,\ldots, T_{n-1}
\]
subject only to the relations
\begin{alignat*}{2}
(T_r-\xi)(T_r+1) &= 0, && \\
T_r T_{r+1} T_r &= T_{r+1}T_rT_{r+1}, && \\
T_r T_s &= T_s T_r &\qquad& \text{if}~|r-s|>1
\end{alignat*}
for all admissible $r$ and $s$.
The algebra 
$\cl H_{n}(\xi)$ is cellular, with cell modules $S^{\la}_{\cl H}$ parameterised by the partitions $\la$ of $n$; see~\cite{M99}.

By the following fundamental results, much of the modular representation theory of Iwahori--Hecke algebras at roots of unity can be phrased in terms of questions about KLR algebras and their universal row Specht modules.

\begin{thm}{\textup{\textbf{\cite[Theorem~1.1]{bkisom},}}}{\textup{\textbf{\cite[Theorem~6.23]{KMR12}.}}}
\label{TF}
Let $K$ be a field of characteristic $p$, and suppose that $\xi \in K \sm \{0\}$
has {\em quantum characteristic} $e$, where either $p=e$ or $p$ is coprime to $e$ if $p\neq 0,e$.
There is an algebra isomorphism $\theta\colon \cl H_{n} (\xi) \iso R^{\La_0}_{n,K}$ such that
if $\cl H_{n} (\xi)$ is identified with $R^{\La_0}_{n,K}$ via $\theta$, then 
the $R^{\La_0}_{n,K}$-module $S^{\la}_{\cl H}$ is isomorphic to $S^{\la}_K$. 
%We have:
%\begin{enumerate}[(i)]
%\item\label{TF1}\textup{\cite[Theorem 1.1]{bkisom}}
%There is an algebra isomorphism 
%$\theta\colon \cl H_{n} (\xi) \iso R^{\La_0}_{n,K}$. 
%\item\textup{\cite[Theorem~6.23]{KMR12}}\label{TF2}
%If $\cl H_{n} (\xi)$ is identified with $R^{\La_0}_{n,K}$ via the isomorphism $\theta$ from (i), then 
%the $R^{\La_0}_{n,K}$-module $S^{\la}_{\cl H}$ is isomorphic to 
%$S^{\la}_K$. 
%\end{enumerate}
\end{thm}

We remark that graded modules over $R^{\La_0}_{n,K}$ which can be identified with the modules $S^{\la}_{\cl H}$ were originally constructed in~\cite{BKW11} and that the proof of the identification of Specht modules in Theorem~\ref{TF} %(ii)
%\eqref{TF2} 
uses results from~\cite{BKW11}. 

Simple $\cl H_n (\xi)$-modules were classified in~\cite{DJ86}, and a corresponding classification of graded simple modules 
$R^{\La_0}_{n,K}$-modules up to isomorphism and grading shift 
is given by~\cite[Theorem 4.11]{BK09}. 
For our purposes, it is convenient to use the description of simple 
$R^{\La_0}_{n,F}$-modules resulting from the 
graded cellular basis of $R^{\La_0}_{n,F}$ constructed by Hu and Mathas~\cite{HM10}, which we now review. 
Let $\la\in \Par(n)$ and define $\mathscr Y^{\la}:=\{ r\in [1,n]\mid \tt^{\la} (r) \in \Z\times e\Z\}$.
%to be the set of
%integers $r\in [1,n]$ such that $(a,b):=(\tt^{\la})^{-1}(r)$ 
%satisfies $e\mid b$. 
Set 
$y^\la:=\prod_{r\in \mathscr Y^\la} y_r\in R^{\La_0}_{n,F}$. 
There is an anti-automorphism $*$ of $R^{\La_0}_{n,F}$ defined on the standard generators by 
\[
1_{\bi}^* = 1_{\bi}, \quad \psi_r^* = \psi_r, \quad y_r^* = y_r. 
\]
For $\st,\tt\in \Std(\la)$, set 
\[
 \psi_{\st\tt}:= \psi_{d(\st)} 1_{\bi^\la} y^\la \psi_{d(\tt)}^* \in R^{\La_0}_{n,F}. 
\]
Then $\psi_{\st\tt}^* = \psi_{\tt\st}$. %Furthermore:
\begin{thmc}{HM10}{Theorem 5.8}
\label{Tcell}
The algebra $R^{\La_0}_{n,F}$ is a graded cellular algebra with weight poset $(\Par(n), \dom)$ and cellular basis 
$\{ \psi_{\st\tt} \mid \st,\tt \in \Std(\la),~\la \in \Par(n) \}$. 
\end{thmc}

Let $\la\in \Par(n)$. 
The cell module corresponding to $\la$ in the graded cellular structure of Theorem~\ref{Tcell} is isomorphic to  $S^\la_F$ as a graded $R^{\La_0}_{n,F}$-module: this follows 
 from~\cite[Corollary~5.10]{HM10} and the proof of~\cite[Theorem~6.23]{KMR12}. In the sequel, we identify $S^\la_F$ with the corresponding graded cell module. 
The cellular structure of Theorem~\ref{Tcell} yields a symmetric bilinear form $\langle \cdot,\cdot\rangle$ on $S^\la_F$. This form is determined by the equation
\begin{equation}\label{Eform}
 \langle v^\tt, v^\st \rangle v^{\tt^\la} = \psi_{\tt^\la\tt} v^{\st} =
 1_{\bi^\la} y^\la \psi_{d(\tt)}^* \psi_{d(\st)} v^{\tt^\la}
 \quad (\tt,\st\in \Std(\la)),
\end{equation}
cf.~\cite[(3.7.2)]{M15}. 
The form $\langle\cdot,\cdot\rangle$ is homogeneous of degree $0$ and satisfies $\langle xu,v\rangle = \langle u,xv\rangle$ for all $x\in R^{\La_0}_{n,F}$ and $u,v\in S^\la_F$, see~\cite[\S2.2]{HM10}.

The radical of  $\langle \cdot,\cdot\rangle$ 
is the $R^{\La_0}_{n,F}$-submodule 
\[
 \rad S^\la_F:= \{ u\in S^\la_F \mid \langle u, v\rangle =0 \text{ for all } v\in S^\la_F \}. 
\]
Define
\[
D^\lam_F:=S^\lam_F /\rad S^\lam_F.
\] 
If $D^\lam_F\neq 0$, then $D^\lam_F$ is an irreducible graded 
$R^{\La_0}_{n,F}$-module. Moreover, by~\cite{DJ86}
and\cite[Corollary~5.11]{HM10}, %we have 
$D^\la_F \ne 0$ if and only if $\la$ is $e$-restricted. Hence, by \cite[Theorem~2.10]{HM10}, we have:

\begin{thm} 
The family $\left\{D^\la_F \mid \la\in \RPar_e(n)\right\}$
is a complete set of graded simple 
$R^{\La_0}_{n,F}$-modules up to isomorphism and grading shift.
\end{thm}

Let $M$ be a finite-dimensional graded $R_{n,F}^{\La_0}$-module. 
Given $\mu\in \RPar_e(n)$,
the {\em graded composition multiplicity} of $D^\mu_F$ in $M$ is 
\[
[M:D^{\mu}_F]_q:= \sum_{k\in \Z} a_k q^k \in \Z[q,q^{-1}], 
\]
where $a_k$ is the multiplicity of $D^\mu_F\langle k\rangle$ in a graded composition series of $M$, see~\cite[\S 2.4]{BK09}.
Since the algebra 
$R^{\La_0}_{n,F}$  is cellular, it is split, so 
\begin{equation}\label{EFK}
[S^\la_F : D^\mu_F ]_q = [S^\la_K : D^\mu_K]_q \qquad (\la\in \Par(n), \, \mu\in \RPar_e (n)). 
\end{equation}
It is well known that $[S^\la_F : D^\mu_F]_q=0$ unless 
$\mu \domby \la$, see e.g.~\cite[Corollary 2.17]{M99}.

\begin{Remark}
Li~\cite{Li14} proved that $R^{\La_0}_{n, \O}$ is cellular for any commutative ring $\O$, generalising Theorem~\ref{Tcell}. Both of these results hold with $\La_0$ replaced by an arbitrary dominant integral weight $\La$, as do the aforementioned results 
from~\cite{bkisom, BKW11, BK09, KMR12}. 
\end{Remark}

%%%%%%%%%% PRELIMINARIES ABOUT PATHS %%%%%%%%%%

\section{Standard tableaux and paths in the weight space}
\label{SPaths}

In this section we fix an integer $k\ge 2$. 
In \S\ref{SSAff}--\ref{SSPaths}, we attach to each standard tableau 
$\tt$ with at most $k$ columns a path $\pi_\tt$ 
in a weight space for the Lie algebra 
$\mathfrak{sl}_k$. We show that the degree of $\tt$ can be non-recursively described in terms of interactions of $\pi_\tt$ with certain hyperplanes in that weight space (Lemma~\ref{lem:deg}) and that the residue sequence of a tableau is invariant under certain reflections of the corresponding path (Lemma~\ref{LResSeq}). In~\S\ref{SSTwo}, we specialise to the case $k=2$, which is the only one used in the rest of the paper, and define a regularisation map on paths, which plays a key role in the sequel.

\subsection{The affine Weyl group}\label{SSAff}
Let $A=(a_{ij})_{1\le i,j\le k-1}$ be the Cartan matrix of finite type $A_{k-1}$, so that 
\[
a_{ij} = \begin{dcases*}
2 & if $i=j$; \\
-1 & if $j=i\pm 1$; \\
0 & otherwise.
\end{dcases*}
\]
Let $\Phi$ be the root system of the Lie algebra $\mathfrak{sl}_k (\C)$ with respect to the Cartan subalgebra $\fra h$ of diagonal matrices in 
$\mathfrak{sl}_k (\C)$, with $\{\alpha_1,\ldots,\alpha_{k-1}\}$ being a set of simple roots. 
We consider the (real) weight space 
$V:=\mR \Phi = \mathfrak{h}_{\mathbb R}^*$, where 
$\mathfrak{h}_{\mathbb R}$ is the 
set of matrices in $\mathfrak h$ with real entries. 
For each $i=1,\ldots,k$, let $\eps_i\in V$ be the weight sending 
a diagonal matrix 
$\diag(t_1,\ldots,t_k)\in \fra h_{\mathbb R}$ to $t_i$. 
Then  $\eps_1+\cdots+\eps_k=0$, and we may assume that $\al_i = \eps_i - \eps_{i+1}$ for $i=1,\ldots,k-1$. 
The set of integral weights is the $\Z$-span $V_{\Z}$ of $\eps_1,\ldots,\eps_k$. 
The set of positive roots is 
$\Phi^+ = \{ \eps_i-\eps_j \mid 1\le i<j\le k\}$. 
Let $(\cdot,\cdot)$ be the symmetric bilinear form $V$ determined by 
$(\al_i,\al_j) = a_{ij}$ for $i,j\in \{1,\ldots,k-1\}$. 
Note that 
\begin{equation}\label{EEpsInn}
(\eps_i,\eps_r-\eps_t)=
\begin{dcases*}
1 & if $i=r$; \\
-1 & if $i=t$; \\
0 & otherwise
\end{dcases*}
\end{equation}
for all $1\le i\le k$ and $1\le r<t\le k$.
The Weyl group $W\cong \Si_k$ of $\Phi$ is the subgroup of 
$\GL (V)$ generated by the simple reflections $s_{\al}$ for $\al\in \Phi$, where
$s_{\al} v = v- (\al,v)\al$ for all $v\in V$. 
The $\Z$-submodule $\Z\Phi\subset V$ is $W$-invariant, and 
the corresponding {\em affine Weyl group} is 
$W_{\aff}:=W\ltimes\mathbb{Z}\Phi$. 
The group $W$ also acts on the set $\{1,\dots,k\}$ in the natural way.

Let $\rho=\sum_{i=1}^{k-1} (k-i) \eps_i$, so that $(\al_i,\rho) = 1$ for all $i=1,\ldots,k-1$. 
There is a faithful action of $W_{\aff}$ on $V$ defined by
\[
w\cdot v:= w(v+\rho)-\rho,\qquad \be\cdot v:=v+e\be\qquad(w\in W,~\be\in\Z\Phi,~v\in V). 
\]
If $\al\in \Phi$ and $m\in \Z$, consider the hyperplane
\[
H_{\alpha,m}:=\{ v\in V \mid(v+\rho,\alpha)=me\} \subset V.
\]
We refer to $H_{\al,m}$ as an \emph{$\al$-wall} or simply a \emph{wall}.
For each $\alpha\in\Phi$ and $m\in\mathbb{Z}$ there exists (unique)
$s_{\alpha,m}\in W_{\aff}$ such that $s_{\al,m}$ acts on $V$ by reflection with respect to $H_{\al,m}$, i.e. 
\begin{equation}\label{Edot}
s_{\alpha,m}\cdot v=v-\big((v+\rho,\alpha)-me\big)\alpha
\end{equation}
for all $v\in V$.
We consider the set
\[
 \mathcal{C}:=\{v\in V\mid(v+\rho,\alpha_i)>0\text{ for }1\leq i<k\},
 \] 
which may be viewed as the dominant chamber of the Coxeter complex corresponding to $\Phi$.
 Using~\eqref{EEpsInn}, we see that
\begin{equation}\label{EC}
\mathcal{C} =\{ c_1\eps_1+\dots +c_k\eps_k \in V \mid
c_1,\dots,c_k\in \mR,\, c_j>c_{j+1}-1 \text{ for all } j=1,\dots,k-1\}.  
\end{equation}

\subsection{Paths}\label{SSPaths}
Let $r,s\in \Z$. 
Recall that $
[r,s]:= \{ t\in \Z \mid r\le t \le s\}.
$
We define $\Pa_{[r,s]}$ to be the set of all maps 
$\pi\colon [r,s]\to V_{\Z}$ such that 
$\pi(a+1)-\pi(a) \in \{\eps_1,\ldots,\eps_k\}$ for all $a\in [r,s-1]$.
Further, set 
\[
\Pa_{[r,s]}^+:=\{ \pi\in \Pa_{[r,s]} \mid \pi(a) \in \cl C \text{ for all } a\in [r,s]\}. 
\]
Let $n\in \mathbb Z_{\geq 0}$. We define
\[
 \Pa_n:= \{ \pi \in \Pa_{[0,n]} \mid \pi(0)=0 \} \quad \text{ and } \quad
\Pa^+_n:=\Pa_n \cap \Pa_{[0,n]}^+, 
\]
so that $\Pa^+_n$ is the set of all maps $\pi \colon [0,n]\to V_\Z \cap \cl C$ such that $\pi(0)=0$ and $\pi(a+1)-\pi(a) \in \{\eps_1,\ldots,\eps_k\}$ for all $a\in [0,n-1]$.

Given $\tt\in \CT_{\le k} (n)$, for any $0\le a\le n$ and $1\le j\le k$,  set 
\[
c_{a,j}(\tt) := | \tt (\{ 1,\dots, a\}) \cap (\Z_{>0} \times \{ j \} ) |,
\]
i.e.~$c_{a,j} (\tt)$ is the number of elements of  $\tt (\{ 1,\ldots,a\})$ in the $j$th column.
Define $\pi_{\tt}\in\Pa_n$ by 
\[
 \pi_{\tt} (a) := c_{a,1}(\tt) \eps_1 + \cdots + c_{a,k}(\tt) \eps_k \qquad (a=0,\ldots,n).
\]
Note that the end-point $\pi_{\tt} (n)$ of  the path $\pi_{\tt}$ depends only on the image of $\tt$ (i.e.~only on the shape of $\tt$ in the case when $\tt$ is a standard tableau).  

\begin{lemma}\label{lem:bij}
The assignment $\tt \mapsto \pi_{\tt}$ is a bijection from $\CT_{\le k} (n)$ onto $\Pa_n$ and restricts to a bijection from $\Std_{\le k} (n)$ onto $\Pa_n^+$. 
\end{lemma}

\begin{proof}
The first assertion of the lemma is clear from the definitions. 
For the second assertion, let $\tt \in \CT_{\le k} (n)$
and observe that $\tt$ is a standard tableau if and only if 
$c_{a,j}(\tt)\ge c_{a,j+1}(\tt)$ for all $a=1,\dots,n$ and all $j=1,\dots, k-1$.
The lemma follows by~\eqref{EC}.
\end{proof}

Let $\pi\in \Pa_n$ and suppose that $\pi(a) \in H_{\al,m}$ for some $\al\in \Phi^+$ and $m\in \Z$. We define the path $s^a_{\al,m} \cdot \pi \in \Pa_n$ by setting 
\[ (s^a_{\al,m} \cdot \pi) (b) :=
\begin{dcases*}
 \pi(b) & for $0\le b\le a$;\\
 s_{\al,m} \cdot \pi(b) & for $a<b\le n$. 
\end{dcases*} \]
That is, $s_{\al,m}^a \cdot \pi$ is obtained by reflecting a `tail' of 
$\pi$ with respect to $H_{\al,m}$.

The residue sequence $\bi^\tt$ of $\tt\in \CT_{\le k}(n)$ may be described as follows: if $1\leq a\le n$, then
\begin{equation}\label{ERes}
 i_a^{\tt} = j - c_{a,j} (\tt) + e\Z,
\end{equation}
where $j\in\{1,\dots,k\}$ is determined by the condition that 
$\tt(a)$ is in the $j$th column. 
The following lemma shows that reflecting a tail of a path as above does not change the residue sequence of the corresponding tableau. 

\begin{lemma}\label{LResSeq}
Let $\al\in\Phi^+$, $m\in\Z$ and $a\in [0,n]$.  
Suppose that $\tt\in\CT_{\le k} (n)$ and $\pi_{\tt} (a) \in H_{\al,m}$. 
If $\st \in \CT_{\le k}(n)$ is 
determined by $\pi_{\st} = s_{\al,m}^a \cdot \pi_{\tt}$,  
then $\bi^\st= \bi^{\tt}$.
\end{lemma}

\begin{proof}
Let $1\le r<t\le k$ be such that $\al=\eps_r-\eps_t$.
Since $\pi_\tt (a)\in H_{\al,m}$, we have 
$c_{a,r}(\tt) - c_{a,t} (\tt) = r-t+me$. 
Clearly, $i_b^\st=i_b^\tt$ for $1\le b\le a$.
Let $b\in [a+1,n]$. For any $j\in \{1,\dots,k\}$, we have 
\[
c_{b,j} (\st) =
\begin{cases}
 c_{b,j} (\tt) & \text{if } j\notin \{r,t\}; \\
 c_{b,t} (\tt) - c_{a,t} (\tt) + c_{a,r} (\tt) = c_{b,t} (\tt) +r-t+me & \text{if } j=r; \\
 c_{b,r} (\tt) - c_{a,r} (\tt) + c_{a,t} (\tt) =c_{b,r}(\tt) +t-r -me  & \text{if } j=t. 
\end{cases}
\]
Now, 
if $j\in \{1,\dots,k\}$ is determined by the condition that 
$\tt (b)$ is in the $j$th column, then 
$\st (b)$ is in the $(s_{\al} j)$th column, and it follows using~\eqref{ERes} that $i^\st_b=i^\tt_b$ in all cases.
\end{proof}

Let $u,v\in V_{\Z}\cap \CC$ 
be such that $v-u=\eps_i$ for some $i\in\{1,\ldots,k\}$. 
For every $\al\in \Phi^+$, set
\begin{equation}\label{Edeg1Step}
\deg_{e,\al}(u,v):=
\begin{dcases*}
1 & if $u\in H_{\al,m}$ and $(v+\rho,\alpha)<me$ for some $m\in \Z_{>0}$;\\
-1 & if $v\in H_{\al,m}$ and $(u+\rho,\alpha)>me$ for some $m\in \Z_{>0}$;\\
0 & otherwise.
\end{dcases*}
\end{equation}
Write $\deg_e (u,v) := \sum_{\al\in \Phi^+} \deg_{e,\al} (u,v)$, and for every 
$\pi\in \Pa^+_{[r,s]}$ define 
\begin{equation}\label{Edeg}
 \deg_e(\pi):=\sum_{a=r}^{s-1}\deg_e \! \big(\pi(a), \pi(a+1)\big). 
\end{equation}

\begin{lemma}\label{lem:deg}
For every $\mtt t\in \Std_{\le k}(n)$, 
we have $\deg_e(\pi_\tt)=\deg_e(\tt)$. 
\end{lemma}

\begin{proof}
Let $\la$ be the shape of $\tt$. Arguing by induction on $n$, we see that it is enough to show that 
$d_{\tt (n)} (\la) = \deg_e (\pi_{\tt} (n-1),\pi_\tt (n))$ when $n>0$.
For $j=1,\dots, k$, let $c_j$ be the size of the $j$th column of $\Y\la$, i.e.~$c_j = c_{n,j} (\tt)$. 
 Let $t$ be such that $\tt (n)=(c_t,t)$. Then $i:=t-c_t+e\Z \in I$ is the residue of $\tt(n)$. 
It is easy to see using the definitions and~\eqref{ERes} that 
\[
d_{\tt (n)} (\la) = \# \{ j\in [1,t-1] \mid j-c_j+e\Z = i+1\}
- \# \{ j\in[1,t-1] \mid j-c_j +e \Z = i \}. 
\]
Since $\pi_\tt (n) = \pi_\tt (n-1)+\eps_t$, we have 
$\deg_{e,\al} (\pi_\tt (n-1),\pi_\tt (n)) = 0$ for all $\al \in \Phi^+$ such that $\al\ne \eps_j - \eps_t$ for any $j=1,\dots,t-1$. 
Moreover, using~\eqref{EEpsInn} we see that for each $j=1,\dots,t-1$,  
\[
\deg_{e,\,\eps_j-\eps_t} (\pi_\tt (n-1),\pi_\tt (n)) = 
\begin{dcases*}
  1 & if $c_j-c_t+t - j \equiv -1 \bmod e$; \\
 -1 & if $c_j - c_t+t-j \equiv 0 \bmod e$; \\
  0 & otherwise.
\end{dcases*}
\]
Since $t-c_t + e\Z = i$, we deduce that 
$d_{\tt(n)} (\la) = \deg_e (\pi_{\tt} (n-1),\pi_\tt (n))$, as required.
\end{proof}

\begin{Remark}
The correspondence between standard tableaux and paths, as above, is considered in~\cite[Section 5]{GW99}. 
The degree function~\eqref{Edeg} is similar to the one defined in~\cite[Definition 1.4]{bcs15} in a somewhat different context. 
Also, consider a path $\pi\in P^+_n$ such that 
$\pi(n)$ does not belong to any wall $H_{\al,m}$ and 
$\pi(a) \notin H_{\al,m} \cap H_{\al',m'}$ for any distinct walls $H_{\al,m}$ and $H_{\al',m'}$ whenever $0\le a<n$. Then one can associate with $\pi$ a {\em Bruhat stroll} as defined in~\cite[\S 2.4]{EW13}, and $\deg_e (\pi)$ is precisely the {\em defect} of the corresponding Bruhat stroll.  
\end{Remark}

\subsection{Two-column tableaux}\label{SSTwo}
In this subsection, we assume that $k=2$, and we again fix $n\in \Z_{\ge 0}$.  
Then $\Phi^+ = \{\al_1\}$, and we write $\al=\al_1$.
We identify $V$ with $\mathbb R$ by sending $\eps_1$ to $1$ and $\eps_2$ to $-1$.
Then $\rho=1$, and a wall $H_{\al,m}$ is the singleton set $\{ me-1\}$.  Furthermore, $\CC = \mathbb{R}_{>(-1)}$
 and 
$$\Pa^+_n = \{\pi\colon [0,n] \rightarrow \mathbb{Z}_{\geq 0} \mid \pi(0)=0 \text{ and } \pi(a+1)=\pi(a)\pm 1 \text{ for all } 0 \leq a <n\}.$$
We write $s_m, s^a_m, H_m$ instead of $s_{\al,m}, s^{a}_{\al,m}, H_{\al,m}$ respectively. 

Let 
\begin{equation}\label{EH}
H=\bigcup_{m\in \Z_{>0}} H_m = \{ me-1 \mid m\in \Z_{>0} \}.
\end{equation}

We define a map $\reg_e\colon \Pa^+_n\to\Pa^+_n$ as follows. Given 
$\pi\in \Pa^+_n$, we set $\reg_e(\pi)=\pi$ if $\pi(a)\notin H$ for all 
$a\in \{0,\ldots,n\}$. Otherwise, let $a\in \{0,\ldots,n\}$  
be maximal such that $\pi(a) \in H$ and, if $m\in \Z_{>0}$ is given by the condition that $\pi(a)\in H_m$, define 
\[
\reg_e(\pi) := 
\begin{dcases*}
\pi & if $\pi(n)\geq me-1$;\\
s^a_m  \cdot \pi & if $\pi(n)<me-1$. 
\end{dcases*}
\] 
Less formally, we consider the last point at which the path $\pi$ meets a wall $H_m$ (if such a point exists) and, if this point is greater than the endpoint of $\pi$, we get
$\reg_e(\pi)$ by reflecting the corresponding `tail' of $\pi$ 
with respect to $H_m$. 
Further, we set
\[
r_e (\pi) := 
\begin{dcases*}
0 & if $\reg_e(\pi)=\pi$;\\
1 & if $\reg_e(\pi)\ne\pi$.
\end{dcases*}
\]
The following is clear from the definitions:
\begin{lemma}\label{LSubtr}
We have $\deg_e (\reg_e(\pi)) = \deg_e (\pi) - r_e (\pi)$ for all $\pi\in \Pa^+_n$.
\end{lemma}

By Lemma~\ref{lem:bij}, there is 
a well-defined map 
$\reg_e\colon \Std_{\le 2} (n)\to \Std_{\le 2} (n)$ 
determined by the condition that $\pi_{\reg_e(\tt)}=\reg_e(\pi_{\tt})$ for all 
$\tt\in \Std_{\le 2} (n)$. 
We also have a map 
$r_e\colon \Std_{\le 2}(n) \to \{0,1\}$ defined by $r_e (\tt) := r_e (\pi_\tt)$.

For any $\la\in \Par_{\le 2} (n)$, set
\[
\DStd_e (\la):= \{ \tt \in \Std(\la) \mid \reg_e (\tt) = \tt \}. 
\]
We refer to the elements of $\DStd_e (\la)$ as {\em $e$-regular} standard tableaux.
The following is easily seen:

\begin{lemma}\label{LPreimage}
Let $\mu = (2^x,1^y)\in \Par_{\le 2}(n)$ and $\st \in \Std(\mu)$. 
Write $y=me-1+j$, 
 where $m\in \Z_{\ge 0}$ and $0\le j<e$. 
If $y \geq e-1$, let 
$a\in \{0,\dots,n\}$ be maximal such that $\pi_\st (a) =me-1$, and let 
$\tt$ be the standard tableau determined by the condition that 
$\pi_{\tt}=s_m^a\cdot \pi_{\st}$.  
Then 
\[
\reg_{e}^{-1} (\st) = 
\begin{dcases*}
 \varnothing & if $\st\notin\DStd_e(\mu)$;\\
 \{\st\} & if $\st\in\DStd_e(\mu)$ and either $y<e$ or $j=0$;\\
 \{\st,\tt\} & if $\st\in\DStd_e(\mu)$, $y\geq e$ and $j>0$.
\end{dcases*}
\]
Moreover, if $y\ge e-1$ then $\Shape(\tt) = (2^{x+j},1^{me-1-j})$. 
\end{lemma}

\begin{Examp} Let $e=4$. A path $\pi\in \Pa^{+}_{19}$
and the path $\reg_4 (\pi)$ are depicted below on the left and right, respectively. 
The weight space $V=\mathbb R$ is identified with a horizontal line, but---for presentation purposes---the height at which $\pi(a)$ is drawn gradually increases as $a=0,\dots,n$ increases. (We use this convention throughout the paper.)
The vertical lines indicate the walls $H_m= \{ 4m-1\}$, $m\in \Z_{\ge 0}$. 
In each picture, the steps from $\pi(a)$ to $\pi(a+1)$ for which 
$\deg_4 (\pi(a), \pi(a+1))$ is $1$ or $-1$ are marked by $+$ or $-$, respectively; these steps are highlighted for clarity. The degree of any unmarked step is $0$.
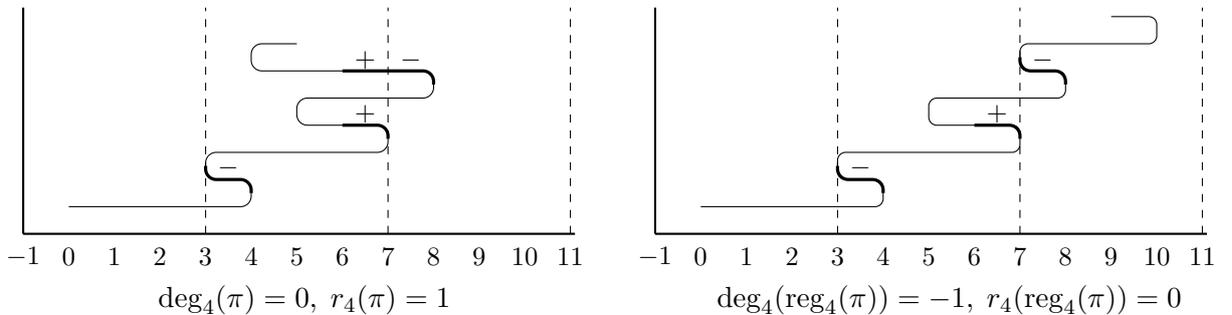
\[
\begin{array}{cc}
\hspace{-2.5mm}
\begin{tikzpicture}[baseline={([yshift=-0.25ex]current bounding box.center)},scale=0.6]
\draw[thick] (-1,0)--(11.1,0);
\draw[thick] (-1,0)--(-1,5);
\foreach \k in {-1,0,1,2,...,10,11} {\node at (\k,-0.5){\small$\k$};};
\foreach \k in {3,7,11} {\draw[dashed](\k,0)--(\k,5);};
\draw[rounded corners=4pt]
(0,0.6)--(4,0.6)--(4,1.2)--(3,1.2)--(3,1.8)--(7,1.8)--(7,2.4)--(5,2.4)--(5,3)--(8,3)--(8,3.6)--(4,3.6)--(4,4.2)--(5,4.2);
\draw[rounded corners=4pt,very thick](4,0.9)--(4,1.2)--(3,1.2)--(3,1.5); 
\node at (3.5,1.45){\small $-$};
\draw[rounded corners=4pt,very thick](7,2.1)--(7,2.4)--(6,2.4); 
\node at (6.5,2.65){\small $+$};
\draw[rounded corners=4pt,very thick](8,3.3)--(8,3.6)--(6,3.6); 
\node at (6.5,3.85){\small $+$};
\node at (7.5,3.85){\small $-$};
\end{tikzpicture}
&
\;
\begin{tikzpicture}[baseline={([yshift=-0.25ex]current bounding box.center)},scale=0.6]
\draw[thick] (-1,0)--(11.1,0);
\draw[thick] (-1,0)--(-1,5);
\foreach \k in {-1,0,1,2,...,10,11} {\node at (\k,-0.5){\small$\k$};};
\foreach \k in {3,7,11} {\draw[dashed](\k,0)--(\k,5);};
\draw[rounded corners=3pt]
(0,0.6)--(4,0.6)--(4,1.2)--(3,1.2)--(3,1.8)--(7,1.8)--(7,2.4)--(5,2.4)--(5,3)--(8,3)--(8,3.6)--(7,3.6)--(7,4.2)--(10,4.2)--(10,4.8)--(9,4.8);
\draw[rounded corners=4pt,very thick](4,0.9)--(4,1.2)--(3,1.2)--(3,1.5); 
\node at (3.5,1.45){\small $-$};
\draw[rounded corners=4pt,very thick](7,2.1)--(7,2.4)--(6,2.4); 
\node at (6.5,2.65){\small $+$};
\draw[rounded corners=4pt,very thick](8,3.3)--(8,3.6)--(7,3.6)--(7,3.9); 
\node at (7.5,3.85){\small $-$};
\end{tikzpicture}
\vspace{1mm}
\\
\deg_4(\pi)=0, \; r_4(\pi)=1 
& 
\hspace{7mm} \deg_4 (\reg_4(\pi))=-1, \; r_4(\reg_4(\pi))=0
\end{array}
\]
\newline
The standard tableau $\tt$ such that $\pi=\pi_\tt$ and the tableau $\reg_4 (\tt)$ are the transposes of 
\[
\begin{ytableau}
 1& 2 & 3 & 4 &6 & 7 & 8 & 9 &12 & 13 & 14 &19  \\
 5 & 10 & 11 & 15 & 16 & 17 & 18
\end{ytableau}
\quad \text{and} \quad 
\begin{ytableau}
 1& 2 & 3 & 4 &6 & 7 & 8 & 9 &12 & 13 & 14 &16 & 17 & 18  \\
 5 & 10 & 11 & 15 & 19
 \end{ytableau}
\]
respectively. 
\end{Examp}

\section{Characters and graded decomposition numbers for 2-column partitions}\label{SComb}

Let $F$ be a field of characteristic $p\ge 0$. We assume that 
$p=0$, $p=e$ or $p$ is coprime to $e$, cf.~\S\ref{SSHecke}. 
We fix  $n\in\mathbb{Z}_{\geq 0}$ and
 use the notation of~\S\ref{SSTwo} throughout the section. 
 In particular, $\Pa^+_n$ is a set of paths in a weight space of type $A_1$, and there is a bijection $\Std_{\le 2} (n) \iso \Pa^+_n$ given by $\tt \mapsto \pi_\tt$, see Lemma~\ref{lem:bij}.

\subsection{$(e,p)$-regularisation}
We now define the {\em $(e,p)$-regularisation map}
$\reg_{e,p}\colon \Pa^+_n\to \Pa^+_n$, which is needed to state the main results of this section. 
If $p=0$, then set $\reg_{e,p}:=\reg_e$. 
If $p>0$, then $\reg_{e,p}$ 
is defined recursively, as follows. For all $\pi\in \Pa^+_n$: 
\begin{enumerate}
\item If $\reg_{ep^z}(\pi)=\pi$ for all $z\in \Z_{\ge 0}$, then set $\reg_{e,p}(\pi):=\pi$. 
\item Otherwise, $\reg_{e,p}(\pi):=\reg_{e,p}(\reg_{ep^z}(\pi))$, 
where $z$ is the largest non-negative integer
such that $\reg_{ep^z}(\pi)\neq\pi$. 
\end{enumerate}
Note that the recursion always terminates because any map $\reg_m$ either fixes a path or increases its end-point.
We also have a map 
\[
\reg_{e,p} \colon \Std_{\le 2} (n) \longrightarrow \Std_{\le 2} (n)
\]
 determined by the identity $\reg_{e,p}(\pi_\tt) = \pi_{\reg_{e,p} (\tt)}$ for all $\tt\in \Std_{\le 2} (n)$. 

Given $\pi\in \Pa^+_n$, we have 
\begin{equation}\label{recur}
\reg_{e,p}(\pi)=\reg_{ep^{z_h}}(\ldots\reg_{ep^{z_2}}(\reg_{ep^{z_1}}(\pi))\ldots)
\end{equation}
where, for each $r=1,\dots,h$, the 
integer $z_r\ge 0$ is maximal such that
$\reg_{ep^{z_{r-1}}} (\ldots \reg_{ep^{z_1}}(\pi)\ldots)$ 
 is not an $ep^{z_r}$-regular path.
Note that $z_1>\cdots >z_h$. 
When $p=0$, we use the convention that $ep^0=e$; in this case, $0\le h\le 1$. 
We refer to~\eqref{recur} as the {\em regularisation equation} and to $Z=\{z_1,\ldots,z_h\}$ as the {\em regularisation set} of $\pi$.
If $\tt\in \Std_{\le 2} (n)$, then the regularisation set of $\tt$ is defined to be that of $\pi_\tt$, and the regularisation equation of $\tt$ is also defined to be that of $\pi_\tt$, with $\pi_\tt$ replaced by $\tt$ on both sides. 

For any $\la\in \Par_{\le 2} (n)$, we set $\DStd_{e,0}(\la):= \DStd_e (\la)$ and, if $p>0$, 
\[
\DStd_{e,p}(\la):= \bigcap_{z\ge 0} \DStd_{ep^z} (\la) = \{ \tt \in \Std(\la) \mid \reg_{e,p} (\tt) = \tt \}. 
\]

\begin{Examp}\label{Ex9tab}
Suppose that $e=p=2$ and $\lambda = (2,2,1,1)$. 
Then $\Std(\la) = \{ \tt_1,\dots,\tt_9\}$, where 
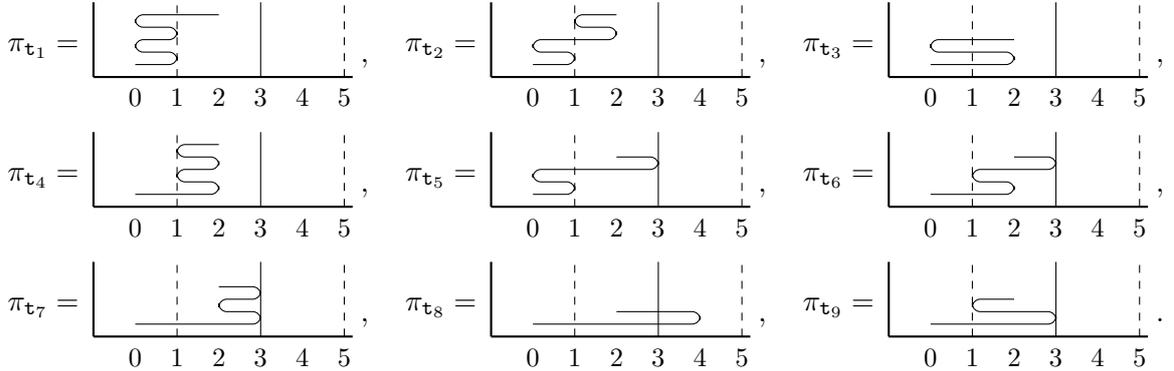
\[
\begin{array}{ccc}
\begin{tikzpicture}[baseline={([yshift=-0.25ex]current bounding box.center)},scale=0.55]
\node at (-2.2,0.4) {$\pi_{\tt_1}=$};
\draw[thick] (-1,-0.3)--(5.2,-0.3);
\draw[thick] (-1,-0.3)--(-1,1.5);
\foreach \k in {0,1,2,...,5} {\node at (\k,-0.8){\small$\k$};};
\foreach \k in {1,5} {\draw[dashed](\k,-0.3)--(\k,1.5);};
\foreach \k in {3} {\draw(\k,-0.3)--(\k,1.5);};
\draw[rounded corners=3pt]
(0,0)--(1,0)--(1,0.3)--(0,0.3)--(0,0.6)--(1,0.6)--(1,0.9)--(0,0.9)--(0,1.2)--(2,1.2);
\end{tikzpicture}, 
& 
\begin{tikzpicture}[baseline={([yshift=-0.25ex]current bounding box.center)},scale=0.55]
\node at (-2.2,0.4) {$\pi_{\tt_2}=$};
\draw[thick] (-1,-0.3)--(5.2,-0.3);
\draw[thick] (-1,-0.3)--(-1,1.5);
\foreach \k in {0,1,2,...,5} {\node at (\k,-0.8){\small$\k$};};
\foreach \k in {1,5} {\draw[dashed](\k,-0.3)--(\k,1.5);};
\foreach \k in {3} {\draw(\k,-0.3)--(\k,1.5);};
\draw[rounded corners=3pt]
(0,0)--(1,0)--(1,0.3)--(0,0.3)--(0,0.6)--(1,0.6)--(2,0.6)--(2,0.9)--(1,0.9)--(1,1.2)--(2,1.2);
\end{tikzpicture},
& 
\begin{tikzpicture}[baseline={([yshift=-0.25ex]current bounding box.center)},scale=0.55]
\node at (-2.2,0.4) {$\pi_{\tt_3}=$};
\draw[thick] (-1,-0.3)--(5.2,-0.3);
\draw[thick] (-1,-0.3)--(-1,1.5);
\foreach \k in {0,1,2,...,5} {\node at (\k,-0.8){\small$\k$};};
\foreach \k in {1,5} {\draw[dashed](\k,-0.3)--(\k,1.5);};
\foreach \k in {3} {\draw(\k,-0.3)--(\k,1.5);};
\draw[rounded corners=3pt]
(0,0)--(1,0)--(2,0)--(2,0.3)--(0,0.3)--(0,0.6)--(1,0.6)--(2,0.6);
\end{tikzpicture}, \vspace{2mm} \\
\begin{tikzpicture}[baseline={([yshift=-0.25ex]current bounding box.center)},scale=0.55]
\node at (-2.2,0.4) {$\pi_{\tt_4}=$};
\draw[thick] (-1,-0.3)--(5.2,-0.3);
\draw[thick] (-1,-0.3)--(-1,1.5);
\foreach \k in {0,1,2,...,5} {\node at (\k,-0.8){\small$\k$};};
\foreach \k in {1,5} {\draw[dashed](\k,-0.3)--(\k,1.5);};
\foreach \k in {3} {\draw(\k,-0.3)--(\k,1.5);};
\draw[rounded corners=3pt]
(0,0)--(1,0)--(2,0)--(2,0.3)--(1,0.3)--(1,0.6)--(2,0.6)--(2,0.9)--(1,0.9)--(1,1.2)--(2,1.2);
\end{tikzpicture}, 
&
\begin{tikzpicture}[baseline={([yshift=-0.25ex]current bounding box.center)},scale=0.55]
\node at (-2.2,0.4) {$\pi_{\tt_5}=$};
\draw[thick] (-1,-0.3)--(5.2,-0.3);
\draw[thick] (-1,-0.3)--(-1,1.5);
\foreach \k in {0,1,2,...,5} {\node at (\k,-0.8){\small$\k$};};
\foreach \k in {1,5} {\draw[dashed](\k,-0.3)--(\k,1.5);};
\foreach \k in {3} {\draw(\k,-0.3)--(\k,1.5);};
\draw[rounded corners=3pt]
(0,0)--(1,0)--(1,0.3)--(0,0.3)--(0,0.6)--(3,0.6)--(3,0.9)--(2,0.9);
\end{tikzpicture},
& 
\begin{tikzpicture}[baseline={([yshift=-0.25ex]current bounding box.center)},scale=0.55]
\node at (-2.2,0.4) {$\pi_{\tt_6}=$};
\draw[thick] (-1,-0.3)--(5.2,-0.3);
\draw[thick] (-1,-0.3)--(-1,1.5);
\foreach \k in {0,1,2,...,5} {\node at (\k,-0.8){\small$\k$};};
\foreach \k in {1,5} {\draw[dashed](\k,-0.3)--(\k,1.5);};
\foreach \k in {3} {\draw(\k,-0.3)--(\k,1.5);};
\draw[rounded corners=3pt]
(0,0)--(2,0)--(2,0.3)--(1,0.3)--(1,0.6)--(3,0.6)--(3,0.9)--(2,0.9);
\end{tikzpicture},  \vspace{2mm} \\
\begin{tikzpicture}[baseline={([yshift=-0.25ex]current bounding box.center)},scale=0.55]
\node at (-2.2,0.4) {$\pi_{\tt_7}=$};
\draw[thick] (-1,-0.3)--(5.2,-0.3);
\draw[thick] (-1,-0.3)--(-1,1.5);
\foreach \k in {0,1,2,...,5} {\node at (\k,-0.8){\small$\k$};};
\foreach \k in {1,5} {\draw[dashed](\k,-0.3)--(\k,1.5);};
\foreach \k in {3} {\draw(\k,-0.3)--(\k,1.5);};
\draw[rounded corners=3pt]
(0,0)--(3,0)--(3,0.3)--(2,0.3)--(2,0.6)--(3,0.6)--(3,0.9)--(2,0.9);
\end{tikzpicture},
& 
\begin{tikzpicture}[baseline={([yshift=-0.25ex]current bounding box.center)},scale=0.55]
\node at (-2.2,0.4) {$\pi_{\tt_8}=$};
\draw[thick] (-1,-0.3)--(5.2,-0.3);
\draw[thick] (-1,-0.3)--(-1,1.5);
\foreach \k in {0,1,2,...,5} {\node at (\k,-0.8){\small$\k$};};
\foreach \k in {1,5} {\draw[dashed](\k,-0.3)--(\k,1.5);};
\foreach \k in {3} {\draw(\k,-0.3)--(\k,1.5);};
\draw[rounded corners=3pt]
(0,0)--(4,0)--(4,0.3)--(2,0.3);
\end{tikzpicture},
&
\begin{tikzpicture}[baseline={([yshift=-0.25ex]current bounding box.center)},scale=0.55]
\node at (-2.2,0.4) {$\pi_{\tt_9}=$};
\draw[thick] (-1,-0.3)--(5.2,-0.3);
\draw[thick] (-1,-0.3)--(-1,1.5);
\foreach \k in {0,1,2,...,5} {\node at (\k,-0.8){\small$\k$};};
\foreach \k in {1,5} {\draw[dashed](\k,-0.3)--(\k,1.5);};
\foreach \k in {3} {\draw(\k,-0.3)--(\k,1.5);};
\draw[rounded corners=3pt]
(0,0)--(3,0)--(3,0.3)--(1,0.3)--(1,0.6)--(2,0.6);
\end{tikzpicture}. 
\end{array}
\]
Here and in the sequel, the thicker the wall $H_m = \{ em-1\}$ is in the diagrams, the greater the $p$-adic valuation of $m$; if this valuation is $0$, the wall is dashed. 

We have $\DStd_{2,2} (\la) = \{ \tt_1,\tt_2,\tt_3,\tt_4\}$. So Theorem~\ref{TMainChar} implies, in particular, that $\dim D^{\la}=4$.
 The regularisation equation of $\pi_{\tt_9}$ is 
 \[
 \reg_{2,2} (\pi_{\tt_9}) = \reg_2 ( \reg_4 (\pi_{\tt_9})),
 \]
 with 
 \[
 \begin{array}{cc}
 \begin{tikzpicture}[baseline={([yshift=-0.25ex]current bounding box.center)},scale=0.55]
\node at (-3.2,0.2) {$\reg_4(\pi_{\tt_9})=$};
\draw[thick] (-1,-0.3)--(6.2,-0.3);
\draw[thick] (-1,-0.3)--(-1,1.2);
\foreach \k in {0,1,2,...,6} {\node at (\k,-0.8){\small$\k$};};
\foreach \k in {1,5} {\draw[dashed](\k,-0.3)--(\k,1.2);};
\foreach \k in {3} {\draw(\k,-0.3)--(\k,1.2);};
\draw[rounded corners=3pt]
(0,0)--(5,0)--(5,0.3)-- (4,0.3);
\end{tikzpicture}
& \quad \text{and} \quad 
\begin{tikzpicture}[baseline={([yshift=-0.25ex]current bounding box.center)},scale=0.55]
\node at (-3.2,0.08) {$\reg_{2,2}(\pi_{\tt_9})=$};
\draw[thick] (-1,-0.3)--(6.2,-0.3);
\draw[thick] (-1,-0.3)--(-1,1.2);
\foreach \k in {0,1,2,...,6} {\node at (\k,-0.8){\small$\k$};};
\foreach \k in {1,5} {\draw[dashed](\k,-0.3)--(\k,1.2);};
\foreach \k in {3} {\draw(\k,-0.3)--(\k,1.2);};
\draw[rounded corners=3pt]
(0,0)--(6,0);
\end{tikzpicture} \;.
 \end{array}
 \]
 Finally, the tableaux $\tt_5,\tt_6,\tt_7,\tt_8$ all have regularisation set $\{ 0\}$, and the images of these tableaux under $\reg_{2,2}$ have shape $(2,1^4)$. 
\end{Examp}

One of the main results of \S\ref{SComb} is the following theorem, which gives a combinatorial description of the graded decomposition numbers 
$[S^\lam:D^\mu]_q$ (when $\lam\in\Par_{\leq 2}(n)$) in terms of the map $\reg_{e,p}$. 

\begin{thm}\label{conj:83}
Let $\lam\in\Par_{\leq 2}(n)$, $\mu\in \RPar_{\le 2}(n)$, and suppose that either $p=e$ or $p$ is coprime to $e$ if $p\neq 0,e$.
If $\mtt s\in\DStd_{e,p}(\mu)$, then
\[
[S^\lam_F:D^\mu_F]_q=\sum_{\mathclap{\substack{\mtt t\in\Std(\lam)\\ \reg_{e,p}(\mtt t)=\mtt s}}} \; q^{r_e (\mtt t)}.
\]
In particular, the right-hand side does not depend on the choice of 
$\mtt s$.
\end{thm}

As we will see, the sum on the right-hand side always contains at most one non-zero term. \smallskip 

In~\S\ref{SSDec}, we give a description of graded decomposition numbers for 2-column partitions that does not use the map $\reg_{e,p}$ and refines a known result on ungraded decomposition numbers, see Theorem~\ref{TDec1}.
In~\S\ref{SSproof}, we use this description to prove 
Theorem~\ref{conj:83}. 

Given $\la,\mu \in \Std_{\le 2} (n)$, define the set 
\begin{equation}\label{EStdepmu}
 \Std_{e,p,\mu} (\la):= \{ \tt \in \Std(\la) \mid \reg_{e,p} (\tt) \text{ has shape } \mu \}. 
\end{equation}
In~\S\ref{SSchar} we show that the subsets $\Std_{e,p,\mu} (\la)$ satisfy the properties stated in \S\ref{SIntro} and use these properties to prove Theorem~\ref{TMainChar}. 

\subsection{Decomposition numbers}\label{SSDec}
Set
\[
D^p (q):= ( [S^\la_F : D^\mu_F]_q)_{\la \in \Par_{\le 2} (n),\,\,\mu\in\RPar_{e,\le 2} (n)},
\]
so that $D^p (q)$ is the submatrix of the graded decomposition matrix of $R^{\La_0}_{n,F}$ 
corresponding to partitions in $\Par_{\leq 2}(n)$.
Let $D^p= D^p(q)|_{q=1}$ denote the corresponding {\em ungraded} submatrix of the decomposition matrix.  

In this subsection, we prove a formula for the entries of $D^p(q)$. 
Recall that for $p>0$, there is a unique square matrix $A^p$---known as an adjustment matrix---such that $D^p = D^0 A^p$. Similarly, there is a unique graded adjustment matrix $A^p(q)$ such that $D^p(q) = D^0(q) A^p(q)$, see~\cite[Theorem~6.35]{M99},\cite[Theorem~5.17]{BK09}.
Write $A^p (q) = (a_{\la\mu}(q))_{\la,\mu\in \RPar_{e,\le 2} (n)}$.
The following is a special case of~\cite[Theorem~5.17]{BK09}:
\begin{thm}\label{TBarInv}
For every $\la \in \Par_{\le 2} (n)$ and $\mu\in \RPar_{e,\le 2} (n)$, we have $\bar a_{\la\mu} (q) = a_{\la\mu} (q)$. 
\end{thm}

Our strategy is as follows.
The ungraded decomposition matrices $D^p$ for all $p > 0$ are known by work of James~\cite{J78,J84} and Donkin~\cite{D98}; 
the graded decomposition matrices $D^0(q)$---and hence the ungraded decomposition matrices $D^0$---are given in~\cite{L06}.  
As we see below, 
all entries in the former matrices are either 0 or 1, so it follows by Theorem~\ref{TBarInv} that $A^p(q)=A^p$, i.e.~every entry of $A^p(q)$ is a constant (Laurent) polynomial. Hence we are able to compute the matrix $D^p(q)=D^0(q)A^p(q)$.  

We begin with some notation.  Given $a,b\in \Z_{\ge 0}$ and assuming that $p>0$, we say that 
\emph{$a$ contains $b$ to base $p$} 
and write $b\preceq_p a$ if, writing
\[a=a_0+a_1p+a_2 p^2+\cdots\quad\text{and}\quad b=b_0+b_1p+b_2 p^2+\cdots\] (with $0\leq a_i,b_i<p$ for all $i$), 
for each $i\ge 0$  either $b_i=0$ or $b_i=a_i$. 
We write $b\preceq_0 a$ whenever $a\ge 0$ and $b=0$. 
  
For any integer $s\geq 0$, let $s_e$ denote the integer part of $s/e$. 
For $p\geq 0$ and non-negative integers $s$ and $l$, define 
\begin{equation}f_{e,p}(l,s):=\begin{dcases*}
1 & if $s_e\preceq_p (l+1)_e$, and either $e\mid s$ or 
$e\mid l+1-s$;\\ 0 & otherwise.\end{dcases*}\label{james_fn}
\end{equation}
We also define a graded version of these numbers: 
\begin{equation}f_{e,p}^q(l,s):=\begin{dcases*}
1& if $s_e\preceq_p (l+1)_e$ and $e\mid s$;\\ 
q & if $s_e\preceq_p (l+1)_e$, $e\mid l+1-s$ and $e\nmid s$;\\ 
0 & otherwise.\end{dcases*}\label{Efq}
\end{equation}

Note that $f_{e,p}^q(l,s)|_{q=1} = f_{e,p} (l,s)$ and $f_{e,p}^q (l,0)=1$.
The ungraded decomposition numbers when $p>0$ are given by the following theorem. 
This theorem was proved by James in the case when $p=e$ 
(see~\cite[Theorem~24.15]{J78}) and---under certain additional conditions---when 
$p\ne e$ (see~\cite[Theorem~20.6]{J84}). The result for $p\ne e$ in full generality follows from a theorem of Donkin~\cite[Theorem 4.4(6)]{D98} together with~\eqref{EFK}.
We follow the statement given by Mathas in~\cite[p.~127]{M99}.

\begin{thm} \label{james_thm}
Assume that $p > 0$ and either $p=e$ or $p$ is coprime to $e$. Suppose that $\lam=(2^u,1^v)\in \Par_{\le 2}(n)$ and 
$\mu=(2^x,1^y)\in \RPar_{e,\le 2}(n)$, with $u\ge x$.
Then
\[[S^\lam_F :D^\mu_F]=f_{e,p}(y,u-x).\]
\end{thm}

Our aim is to prove that  $[S^\la_F :D^\la_F]_q = f_{e,p}^q (y,u-x)$ under the hypotheses of Theorem~\ref{james_thm}. 
We do so with the aid of 
alternative descriptions of $f_{e,p}^q (y,u-x)$, which are given by Lemmas~\ref{LDec} and~\ref{altform} 
(for $p>0$ and $p=0$ respectively)
and are used in the rest of the paper. 
First, we note the following elementary fact:

\begin{lemma}\label{Lsigns_unique}
Assume that $p>0$. Let $a_0,\dots,a_s,l\in \Z_{\ge 0}$ for some $s\in \Z_{\ge 0}$ satisfy $0\le a_i<p$ for 
$i=0,\dots,s$ and $0\le l<e$. If 
$\delta_0,\dots,\delta_{s+1},\delta'_0,\dots,\delta'_{s+1}\in \{0,1\}$ and 
\begin{equation}\label{Esu1}
\Big(\sum_{i=0}^s (-1)^{\delta_{i+1}} a_i p^i \Big)e + (-1)^{\delta_0} l 
=
\Big(\sum_{i=0}^s (-1)^{\delta'_{i+1}} a_i p^i \Big)e + (-1)^{\delta'_0}l,
\end{equation}
then 
\begin{enumerate}[(i)]
\item 
$\delta_{i+1} = \delta'_{i+1}$ whenever $0\le i\le s$ and $a_i\ne 0$; 
\item
 $\delta_0=\delta'_0$ if $l\ne 0$. 
 \end{enumerate}
\end{lemma}

\begin{proof}
Assume that (i) is false, and let $i$ be the largest index such that 
$\delta_{i+1}\ne \delta'_{i+1}$ and $a_i> 0$. Without loss of generality, $\delta_{i+1}=0$ and $\delta'_{i+1} =1$, and hence the difference between the left-hand and the right-hand sides of~\eqref{Esu1} is at least 
\[
2 a_i p^i e - \sum_{j=0}^{i-1} 
2 a_j p^j e - 2 l \ge 
2 ep^i  -  2e(p-1)\sum_{j=0}^{i-1} p^j  - 2 (e-1)>0.
\]
This contradiction proves (i), and (ii) follows immediately. 
\end{proof}

\begin{defn}
Let $y\in \Z_{\ge 0}$. If $p>0$, then we define the {\em $(e,p)$-expansion} of $y$ to be the expression 
\[
 y=(a_sp^s+a_{s-1}p^{s-1}+\cdots+a_1 p+a_0)e+l-1,
\]
where $0\leq l<e$, $0\leq a_i<p$ for all $i$, and if $y \geq e-1$ then $a_s>0$ 
(cf.~\cite[\S 3.4]{D98}).
\end{defn}

%The $(e,p)$-expansion of a nonnegative integer is unique up to adding or removing finitely many zero terms at the front. 

\begin{lemma}\label{LDec}
Assume that $p>0$. 
Let $\mu = (2^x,1^y)\in \RPar_{e,\le 2}(n)$ and $\la=(2^u,1^v)\in\Par_{\le 2}(n)$. Suppose that 
\[
y=(a_sp^s+a_{s-1}p^{s-1}+\cdots+a_1 p+a_0)e+l-1 
\]
is the $(e,p)$-expansion of $y$. 
Then $f_{e,p}^q (y,u-x)\ne 0$  if and only if $v$ is of the form 
\[v=(a_sp^s\pm a_{s-1}p^{s-1}\pm\cdots\pm a_1p\pm a_0)e\pm l-1\] for some choice of signs. Moreover, in this case 
\begin{equation}\label{EDec1}
f_{e,p}^q (y,u-x)=\begin{dcases*} 1 & if $l=0$;\\
1 & if $0<l<e$ and $v=(a_sp^s\pm a_{s-1}p^{s-1}\pm\cdots\pm a_0)e+l-1$;\\
q & if $0<l<e$ and $v=(a_sp^s\pm a_{s-1}p^{s-1}\pm\cdots\pm a_0)e-l-1$.
\end{dcases*}
\end{equation}
\end{lemma}

We note that the last two cases in~\eqref{EDec1} cannot occur simultaneously by Lemma~\ref{Lsigns_unique}. 

\begin{proof}
We may assume that $u \geq x$.  
If $y<e$, then $f_{e,p}^q (y,u-x)=0$ unless $u=x$. So we may assume that $y\ge e$. 
Writing \[u-x=(b_sp^s+b_{s-1}p^{s-1}+\cdots +b_1p+b_0)e+k,\] where $0\leq k<e $ and $0\leq b_i<p$ for all $i$, we have \[v=y-2(u-x)=\big((a_s-2b_s)p^s+(a_{s-1}-2b_{s-1})p^{s-1}+\cdots +(a_0-2b_0)\big)e+(l-2k)-1.\] 
Moreover, 
\begin{enumerate}
\item $(u-x)_e\preceq_p (y+1)_e$ if and only if 
$b_s=0$ and $b_i\in \{0, a_i\}$ for all $0\leq i<s$;
\item $e\mid u-x$ if and only if $k=0$;
\item $e\mid y+1-u+x$ if and only if $k=l$.
\end{enumerate}
The result now follows from~\eqref{Efq}. 
\end{proof}

\begin{lemma} \label{altform}
Let $\lam=(2^u,1^v)\in \Par_{\le 2} (n)$ and $\mu=(2^x,1^y)\in \RPar_{e,\le 2}(n)$, with $u\ge x$. 
\begin{enumerate}[(i)]
\item We have
 \[f_{e,0}^q(y,u-x)=\begin{dcases*}
1& if $u-x=0$;\\ 
q & if $0<u-x<e \leq y+1$ and $e \mid y+1-u+x$; \\
0 & otherwise.\end{dcases*}\]
\item  The equality $f_{e,0}^q(y,u-x)=q$ holds if and only if there exists an integer $m\geq 1$ such that $y=me+j-1$ and $v=me-j-1$ for some $1\leq j<e$.  
\end{enumerate}
\end{lemma}

\begin{proof}
Part (i) follows from~\eqref{Efq}, and (ii) follows easily from (i) because $e\mid y+1-u+x$ if and only if $2e \mid y+v+2$.
\end{proof}

If $p=0$, the 2-column graded decomposition numbers are given 
in~\cite[Theorem~3.1]{L06}, if that result is interpreted in 
view of~\cite[Corollary 5.15]{BK09}:

\begin{thm} \label{p0_thm}
Suppose that $p=0$. If $\lam=(2^u,1^v) \in \Par_{\leq 2}(n)$ and $\mu=(2^x,1^y) \in \RPar_{e,\leq 2}(n)$ are such that $u\ge x$
then \[[S^\lam:D^\mu]_q =f^q_{e,0}(y,u-x).\]
\end{thm}

The notation in~\cite[Theorem~3.1]{L06} is different from that of Theorem~\ref{p0_thm} or Lemma~\ref{altform}, but it is not difficult to see that the results are equivalent.  
Alternatively, a direct proof of Theorem~\ref{p0_thm} can be easily obtained via the fact that the decomposition numbers 
$[S^\lam_F:D^\mu_F]_q$ are certain parabolic Kazhdan--Lusztig polynomials, see~\cite{VV99} or \cite[Theorem 5.3]{GW99}, together 
with Soergel's algorithm~\cite{S97} for computing those polynomials. 

For $p>0$, define a matrix $\tilde{A}^p=\big(\tilde a^p_{\lam\mu}\big)$ with rows and columns indexed by  $\RPar_{e,\leq 2}(n)$ as follows.
Let $\lam=(2^u,1^v), \mu=(2^x,1^y) \in \RPar_{e,\le 2}(n)$ and let 
\begin{equation*} 
y=(a_sp^s+a_{s-1}p^{s-1}+\cdots+a_1 p+a_0)e+l-1  
\end{equation*}
be the $(e,p)$-expansion of $y$. 
Then we set 
\begin{equation*}
\tilde a^p_{\lam \mu}:=
\begin{dcases*}
1 & if $v=(a_sp^s\pm a_{s-1}p^{s-1}\pm\cdots\pm a_1p\pm a_0)e+l-1$ for some choice of signs;\\
0 & \text{otherwise}.  
\end{dcases*} 
\end{equation*}

\begin{lemma} \label{Adj}
Let $p>0$ and $B=(b_{\lam\mu})_{\la\in \Par_{\le 2} (n), \, \mu \in \RPar_{e,\le 2}(n)}= D^0(q)\tilde{A}^p$.  If $\lam=(2^u,1^v)\in \Par_{\le 2}(n)$ and $\mu=(2^x,1^y)\in \RPar_{e,\le 2}(n)$, with $u\ge x$, then \[b_{\lam\mu}=f^q_{e,p}(y,u-x).\]  
\end{lemma}

\begin{proof}
In view of Lemma~\ref{Lsigns_unique}, 
the result follows from Lemmas~\ref{LDec} and~\ref{altform} together 
with Theorem~\ref{p0_thm}.
\end{proof}

\begin{thm}\label{TDec1} 
Suppose that either $p=e$ or $p$ is coprime to $e$ if $p\neq 0,e$.
If $\lam=(2^u,1^v)\in \Par_{\le 2} (n)$ and $\mu=(2^x,1^y)\in \RPar_{e,\le 2}(n)$ with $u\ge x$, then 
\[[S^\lam_F:D^\mu_F]_q=f_{e,p}^q(y,u-x).\]
\end{thm}

\begin{proof}
If $p=0$, then this is Theorem~\ref{p0_thm}, so we assume that $p>0$.  
Setting $q=1$ in Lemma~\ref{Adj}, we have $D^p=D^0 \tilde{A}^p$, so that $\tilde{A}^p=A^p$ is the $\RPar_{e,\le 2}(n)\times \RPar_{e,\le 2}(n)$-submatrix of 
the ungraded adjustment matrix corresponding to 2-column partitions.  Since the entries of $A^p$ are either $0$ or $1$, it follows from Theorem~\ref{TBarInv} that 
$\tilde{A}^p=A^p=A^p(q)$, that is, $D^p(q)=D^0(q)\tilde{A}^p$.  The result then follows by another application of Lemma~\ref{Adj}.  
\end{proof}

\begin{Examp}
Suppose that $e=p=2$ and $\la=(2,2,1,1)$. 
We have $f_{2,2}^q (2,0)=1 = f_{2,2}^q (6,2)$ and $f_{2,2}^q (4,1)=q$, so by Theorem~\ref{TDec1}
the composition factors of $S^\la$ are $D^{\la}$, $D^{(2,1^4)}$, $D^{(1^6)}$ and their graded composition multiplicities are $1,q,1$ respectively. The reader may wish to check, using Example~\ref{Ex9tab}, that Theorem~\ref{conj:83} holds in this case. 
\end{Examp}

\subsection{Proof of Theorem~\ref{conj:83}}\label{SSproof}
We fix a partition $\mu=(2^x,1^y)\in \RPar_{e,\le 2}(n)$ 
and prove Theorem~\ref{conj:83} for this fixed $\mu$.
Recall the sets $\Std_{e,p,\mu}(\lam)$ for 
$\la\in \Par_{\le 2}(n)$ 
defined by~\eqref{EStdepmu}.
If $p=0$, then the statement of Theorem~\ref{conj:83} is an immediate consequence of Lemma~\ref{LPreimage},  Lemma~\ref{altform} and Theorem~\ref{p0_thm}. 
If $p>0$, then we set 
\begin{equation*} 
y=(a_sp^s+a_{s-1}p^{s-1}+\cdots+a_1 p+a_0)e+l-1 
\end{equation*}
to be the $(e,p)$-expansion of $y$.

\begin{lemma} \label{ExpDesc}
Assume that $p>0$. 
Let $\lam=(2^u,1^v)$ and let 
$\mtt t\in\Std_{e,p,\mu}(\lam)$.
If \[\reg_{e,p}(\mtt t)=\reg_{ep^{z_h}}(\ldots\reg_{ep^{z_2}}(\reg_{ep^{z_1}}(\mtt t))\ldots)\]
is the regularisation equation of $\tt$, then $a_{z_k}> 0$ for all $1\leq k\leq h$ and 
\[v=\left(a_sp^s+(-1)^{\delta_s}a_{s-1}p^{s-1}+\cdots+(-1)^{\delta_2}a_1p+(-1)^{\delta_1} a_0\right)e+(-1)^{\delta_0}l-1,\]
where $\delta_i=\#\{1\leq k\leq h\mid z_k\geq i\}$ for all $0\leq i\leq s$.        
\end{lemma}

\begin{proof}
We use induction on $h$, the result being clear when $h=0$. Assuming that $h>0$, set 
$\rt=\reg_{ep^{z_1}}(\tt)$ so that 
\[\reg_{e,p}(\tt)=\reg_{ep^{z_h}}(\ldots\reg_{ep^{z_3}}(\reg_{ep^{z_2}}(\rt))\ldots)\]
is the regularisation equation of $\rt$, 
and set $\nu:=\Shape(\rt)=\big(2^c,1^d\big)$. By the inductive hypothesis, 
\begin{multline*}
d=\left(a_s p^s + a_{s-1} p^{s-1} +\cdots+ a_{z_2}p^{z_2} 
+ (-1)^{\epsilon_{z_2}}a_{z_2-1} p^{z_2-1} +\cdots +  (-1)^{\epsilon_1} a_0\right) \! e + (-1)^{\epsilon_0} l -1,
\end{multline*}
where $a_{z_2},\ldots,a_{z_h}>0$ and $\epsilon_i=\#\{2\leq k\leq h\mid z_k\geq i\}$ for all $0\leq i\leq z_2$.

Let $m=a_sp^s+\cdots +a_{z_1}p^{z_1}$. Then $d=me+j-1$ for some $0\le j<ep^{z_1}$.
Since $\rt=\reg_{ep^{z_1}}(\tt)\neq\tt$, 
we have 
\begin{multline*}
v=2(me-1)-d=\big(a_sp^s+a_{s-1}p^{s-1}+\cdots+a_{z_1}p^{z_1}-a_{z_1-1}p^{z_1-1}-\cdots -a_{z_2}p^{z_2}\\ 
-(-1)^{\epsilon_{z_2}}a_{z_2-1}p^{z_2-1}-\cdots -(-1)^{\epsilon_1}a_0\big)e-(-1)^{\epsilon_0}l-1
\end{multline*}
by Lemma~\ref{LPreimage}. 
Also, $a_{z_1}>0$, for if $a_{z_1}=0$ 
then 
either $\tt \notin \DStd_{ep^{z_1+1}}(\la)$ or $\tt\in \DStd_{ep^{z_1}} (\la)$, contradicting the hypothesis on the regularisation equation of $\tt$. 
The result follows. 
\end{proof}

The following lemma is an easy consequence of the definitions in~\S\ref{SSTwo}.

\begin{lemma}\label{LRegPres}
Assume that $p>0$ and let 
$b \in\Z_{\geq 0}$. 
\begin{enumerate}[(i)]
\item
Let $\tt\in \Std_{\le 2}(\la)$ and $\st:=\reg_{ep^b} (\tt)$, and suppose that $b$ is maximal such that $\st=\reg_{ep^b} (\tt)$.  Let $c>b$ be an integer. 
Then $\tt$ is $ep^c$-regular if and only if $\st$ is $ep^c$-regular. 
\item 
Let $\tt\in \Std_{\le 2}(n)$. Suppose that 
$\st:=\reg_{ep^b} (\tt) \ne \tt$ and $e\nmid \pi_{\tt} (n)+1$. Then $\tt$ is $e$-regular  if and only if $\st$ is not $e$-regular.
\end{enumerate}
\end{lemma}

\begin{lemma} \label{existence}
Let $\lam=(2^u,1^v)\in\Par_{\leq 2}(n)$ be such that $S^\lam_F$ has a composition factor $D^\mu_F$.  
\begin{enumerate}[(i)]
\item
If $\st\in\DStd_{e,p}(\mu)$ 
 then there exists a unique $\tt\in\Std(\lam)$ 
 such that $\reg_{e,p}(\tt)=\st$.
 \item
 All elements of $\Std_{e,p,\mu} (\la)$ have the same regularisation set. 
 \end{enumerate}
\end{lemma}

\begin{proof}
If $p=0$, then the result holds by Lemma~\ref{LPreimage}. So, assume that $p>0$. 

By Theorem~\ref{TDec1} and Lemma~\ref{LDec},
\[v=\left((-1)^{\delta_{s+1}}a_s p^s + (-1)^{\delta_{s}} a_{s-1} p^{s-1} +\cdots + (-1)^{\delta_1} a_0\right)e + (-1)^{\delta_0} l -1\] 
for some $\delta_0,\dots,\delta_{s+1}\in \{0,1\}$ satisfying the following conditions:
\begin{enumerate}
\item $\delta_{s+1}=0$;
\item $\delta_0=\delta_1$ if $l=0$, and 
$\delta_{i+1} = \delta_{i+2}$ if $a_i=0$ and $0\le i< s$.
\end{enumerate}
By Lemma~\ref{Lsigns_unique}, such $\delta_0,\dots,\delta_{s+1}$ are uniquely determined by $v$.
Let 
\[
Z:=\{i\in [0,s] \mid \delta_i \ne \delta_{i+1}\}= \{ z_1>\dots >z_{h}\}.
\] 
By Lemma~\ref{ExpDesc}, if $\tt\in \Std_{e,p,\mu} (\la)$, then 
the regularisation set of $\tt$ is $Z$, proving (ii).

To prove (i), we argue by induction on $h$. If $h=0$, then $\la=\mu$ and there is nothing to prove. So, assume that $h>0$. 
Set 
\begin{multline*} 
d=\Big(a_s p^s+a_{s-1}p^{s-1}+\cdots +a_{z_1} p^{z_1}\\
-(-1)^{\delta_{z_1}} a_{z_1-1}p^{z_1-1}-\cdots -(-1)^{\delta_2}a_1p-(-1)^{\delta_1} a_0\Big)e-(-1)^{\delta_0}l-1 
\end{multline*}
and let 
$\nu:=\big(2^c,1^d\big)\in\Par_{\leq 2}(n)$, with $c=(n-d)/2$. 
 By Lemma~\ref{LDec} and Theorem~\ref{TDec1}, $D^\mu_F$ is a composition factor of $S^\nu_F$, so by the inductive hypothesis, there is a unique $\rt\in\Std(\nu)$ such that $\reg_{e,p}(\rt)=\st$.
  Moreover, by Lemmas~\ref{ExpDesc} and~\ref{Lsigns_unique}, 
 \[
  \reg_{e,p} (\rt) = \reg_{ep^{z_h}} (\dots \reg_{ep^{z_2}}(\rt) \dots)
 \]  
 is the regularisation equation of $\rt$. In particular, $\reg_{ep^z} (\rt)=\rt$ for all $z>z_2$. 
Let 
\[
m=a_s p^{s-z_1}+a_{s-1} p^{s-1-z_1}+\dots +a_{z_1}.
\]
Since $\delta_{z_1}\neq \delta_{z_1+1}=\delta_{s+1}=0$, we have $\delta_{z_1}=1$ and so $mep^{z_1}-1 <d<(m+1)ep^{z_1}-1$. 
By Lemma~\ref{LPreimage}, 
\begin{equation}\label{ERegInv}
\reg_{ep^{z_1}}^{-1} (\rt) = \{ \rt, \tt\}
\end{equation} 
for a certain tableau $\tt\in \Std(\la)$. 
By Lemma~\ref{LRegPres}(i), $\reg_{ep^z} (\tt) = \tt$ for all $z>z_1$.
Hence, $\reg_{e,p} (\tt) = \reg_{e,p} (\rt)=\st$.

To prove the uniqueness statement in (i), recall that any  $\tt'\in \Std(\la)$ such that $\reg_{e,p} (\tt') = \st$ has regularisation set 
$Z$.
By the inductive hypothesis, this implies that $\reg_{ep^{z_1}} (\tt') = \rt$, and by~\eqref{ERegInv} we have $\tt'=\tt$. 
\end{proof}

\begin{lemma} \label{DecNo}
Let $\lam=(2^u,1^v)$ be a partition of $n$ such that $D^\mu_F$ is a composition factor of $S^\la_F$. 
If $\mtt t\in\Std_{e,p,\mu}(\lam)$ then
$[S^\lam_F:D^\mu_F]_q=q^{r_e(\tt)}$.   
\end{lemma}

\begin{proof}
If $p=0$, then the result holds by Theorem~\ref{p0_thm} and Lemma~\ref{altform}, so we assume that $p>0$. 
Consider the regularisation equation of $\tt$:
\[\st:=\reg_{e,p}(\tt)=\reg_{ep^{z_h}}(\ldots\reg_{ep^{z_2}}(\reg_{ep^{z_1}}(\tt))\ldots).\]
By Lemma~\ref{ExpDesc}, 
\[v=\left (a_s p^s + (-1)^{\delta_{s}} a_{s-1} p^{s-1} +\cdots+(-1)^{\delta_2}a_1 p + (-1)^{\delta_1} a_0\right)\! e + (-1)^{\delta_0} l -1,\] where $\delta_i = \#\{1\leq k\leq h\mid z_k\geq i\}$ for $0\leq i\leq s$. 
Noting that $\st \in \DStd_e (\mu)$ and applying Lemma~\ref{LRegPres}(ii) $h$ times, we see that 
\[r_e(\tt) = \begin{dcases*} 
0 & if $l=0$;\\ 
0 & if $l>0$ and $h$ is even;\\
1 & if $l>0$ and $h$ is odd.
\end{dcases*}\]
Observe that $\delta_0=h$. 
By Theorem~\ref{TDec1} and Lemma~\ref{LDec}, the result follows.
\end{proof}

\begin{proof}[Proof of Theorem~\ref{conj:83}] 
By Theorem~\ref{TDec1} and Lemmas~\ref{LDec} and~\ref{ExpDesc} (or by Lemma~\ref{LPreimage} and Theorem~\ref{p0_thm} in the case $p=0$), if 
$\Std_{e,p,\mu} (\la) \ne \varnothing$ 
 then $D^\mu_F$ is a composition factor of $S^\la_F$. 
The converse also holds, by Lemma~\ref{existence}(i). Moreover, that lemma asserts that, 
given $\st\in \DStd_{e,p} (\mu)$, there exists a 
unique $\tt\in \Std(\la)$ such that $\reg_{e,p} (\tt) = \st$. Further, by Lemma~\ref{DecNo}, we have $[S^\la_F:D^\mu_F]_q = q^{r_e (\tt)}$ in this case, completing the proof.
\end{proof}

%%%%%%%%%% FORMAL CHARACTERS %%%%%%%%%%

\subsection{Characters of simple modules}\label{SSchar}
Recall the notation of~\S\S\ref{SSPaths}-\ref{SSTwo} applied to a weight space of type $A_1$. 
Let $u\le r\le s\le v$ 
be integers and $\pi\in\Pa^+_{[u,v]}$. 
Denote 
the restriction of $\pi$ to $[r,s]$ by $\pi[r,s]$, 
so that $\pi[r,s]\in \Pa^+_{[r,s]}$. 
We refer to every such restriction $\pi[r,s]$ as a {\em segment} of 
$\pi$. The following is clear:

\begin{lemma}\label{LDegSum}
If $u=c_1 \le \cdots \le c_l = v$ are integers, then $\deg_e (\pi) = \sum_{i=1}^{l-1} \deg_e (\pi[c_i,c_{i+1}])$. 
\end{lemma}

\begin{defn}
Let $r<s$ be integers.
Let $\eta\in \Pa^+_{[r,s]}$ be such that 
$\eta(r)=\eta(s)=me-1$ for some $m\in \Z_{>0}$, so that $\eta(r)\in H_m$. 
\begin{itemize}
\item If $me-1<\eta(c)<(m+1)e-1$
whenever $r<c<s$, then we call $\eta$ 
a \emph{positive arc}.
\item
If $(m-1)e-1<\eta(c)<me-1$ whenever $r<c<s$, then we call $\eta$
a \emph{negative arc}.
\end{itemize}
If one of the above two statements holds, we say that $\eta$ is an {\em arc}.
\end{defn}

Let $m\in \Z_{>0}$ and $\pi \in P^+_{[u,v]}$. Recall the reflections $s_m = s_{\al,m}$ 
from~\eqref{Edot}.
 We define 
$s_m \cdot \pi\in \Pa_{[u,v]}$ to be the path obtained by reflecting $\pi$ with respect to the wall $H_m$, so that 
\[
 (s_m \cdot \pi) (a) = s_m\cdot (\pi(a)) \quad (u\le a\le v).
\]

\begin{lemma}\label{LReflResSeq}
Let $\tt \in \Std_{\le 2} (n)$. 
Suppose that $0\le r<s \le n$ are integers and 
$\pi_{\tt}[r,s]$ is an arc, 
with $\pi_{\tt} (r) = \pi_{\tt} (s) = me-1$ where 
$m\in \Z_{>0}$. If $\st\in \Std_{\le 2} (n)$ is defined by the conditions that 
\[
\pi_\st[0,r] = \pi_\tt[0,r], \quad \pi_\st[r,s] = s_m \cdot \pi_\tt[r,s], \quad 
\pi_\st[s,n] = \pi_\tt [s,n],
\]
then $\bi^{\st} = \bi^{\tt}$. 
\end{lemma}

\begin{proof}
Since $\pi_{\st} = s_m^{s} \cdot s_m^r \cdot \pi_\tt$, the result follows from Lemma~\ref{LResSeq}. 
\end{proof}

If $\pi \in P^{+}_{[u,v]}$, we define $\mathsf A^+(\pi)$ and $\mathsf A^-(\pi)$ to be the sets of segments of $\pi$ that are positive and negative arcs, respectively.
Recalling~\eqref{EH}, let 
\begin{equation}\label{Ebs}
B(\pi)=\{ b_1<\cdots<b_N \}:= \{ b\in [u,v] \mid \pi(b) \in H\}.
\end{equation}
Then every element of $\mathsf A^+(\pi) \cup \mathsf A^-(\pi)$ is of the form $\pi[b_k,b_{k+1}]$ for some $1\le k<N$. 

\begin{lemma}\label{LDegArcs1} 
Let $\pi\in \Pa^+_{[u,v]}$ be such that $\pi(u),\pi(v)\in H$ and let $B(\pi)=\{b_1<\dots<b_N\}$. Then:
\begin{enumerate}[(i)]
\item 
 $\deg_e (\pi) = |\Arc^- (\pi)| - |\Arc^+ (\pi)|$. 
 \item 
 If $m\in \Z_{>0}$ and $s_m \cdot \pi\in \Pa^+_{[u,v]}$, then 
 $\deg(s_m \cdot \pi)=-\deg(\pi)$. 
 \end{enumerate}
\end{lemma}

\begin{proof}
It is easy to see from~\eqref{Edeg1Step} and~\eqref{Edeg} that, whenever $1\le k<N$, 
\[
\deg_e (\pi[b_k, b_{k+1}]) =
\begin{dcases*}
1 & if $\pi[b_k,b_{k+1}]$ is a negative arc; \\
-1 & if $\pi[b_k,b_{k+1}]$ is a positive arc; \\
0 & if $\pi(b_k)\neq\pi(b_{k+1})$. 
\end{dcases*}
\] 
Part (i) follows by Lemma~\ref{LDegSum}. Part (ii) follows from (i) because the reflection $s_m$ transforms positive arcs into negative arcs and vice versa. 
\end{proof}

The following result describes the degree of a path in terms of the number of its positive and negative arcs.

\begin{corollary} \label{Cdeg_arcs}
If $\pi\in \Pa^+_n$, then 
$\deg_e(\pi)=|\Arc^-(\pi)|-|\Arc^+(\pi)|+r_e(\pi)$. 
\end{corollary}

\begin{proof}
Let the integers $0<b_1< \dots < b_N\le n$ be defined as in~\eqref{Ebs}. Then $\deg_e (\pi[0,b_1])=0$.  Furthermore, 
$\deg_e (\pi[b_N,n])=r_e(\pi)$.  
The result follows from Lemmas~\ref{LDegSum} and~\ref{LDegArcs1}(i). 
\end{proof}

Our next goal is to describe the effect of the map $\reg_{e,p}$ on the degree of a standard tableau. If $p>0$, then given $\eta\in \Pa^+_n$ and a finite subset $Z=\{ z_1>\cdots >z_h\}$ of $\Z_{\ge 0}$, 
we define the tuple 
\[
\bw(Z,\eta) = (w_1,\dots,w_h) \in [0,n]^{h}
\] by setting $w_i$ to be the maximal element of $[0,n]$ such that 
$\eta(w_i) \in \bigcup_{m\in \Z_{>0}} H_{mp^{z_i}}$, for each $i=1,\dots, h$, with $w_i=0$ if no such element exists.   

\begin{lemma} \label{LDegReg}
Assume that $p>0$. 
Let $\pi\in\Pa^+_n$ and $\eta= \reg_{e,p} (\pi)$. Suppose that $Z$ is the regularisation set of $\pi$ and 
$\bw(Z,\eta) = (w_1,\dots,w_h)$. 
Let  $w_0=0$ and $w_{h+1}\in [0,n]$
be maximal such that $\pi(w_{h+1}) \in H$, with $w_{h+1}=0$ if no such number exists. 
Then
\[
\deg_e (\eta[w_i,w_{i+1}]) = (-1)^i \deg_e (\pi[w_i,w_{i+1}]). 
\]
for all $0\le i\le h$. 
\end{lemma}

\begin{proof}
Let $Z=\{ z_1>\cdots > z_h\}$ and, for any $0\le i\le h$, set 
\[
\pi_i:=\reg_{ep^{z_i}}( \dots ( \reg_{ep^{z_1}} (\pi) ) \dots ),
\]
so that $\pi_0=\pi, \pi_1,\dots, \pi_h=\eta$ are the `intermediate steps' in the calculation of $\reg_{e,p} (\pi)$. 
Given $1\le i\le h$, let $m_i\in p^{z_i} \Z_{>0}$ be defined by the condition that $\pi_{i-1} (w_i) = m_i e-1$. Then
\[
\eta[w_i,w_{i+1}] = s_{m_i} \cdot \ldots \cdot s_{m_1} \cdot \pi[w_i,w_{i+1}].
\]
whenever $0\le i\le h$.
Using Lemma~\ref{LDegArcs1}(ii), we deduce the required identity.
\end{proof}

Let $Z=\{ z_1>\cdots > z_h\}$ be a finite subset of $\Z_{>0}$. 
We define a map $\rho_Z \colon \Pa^+_n \to \Pa^+_n$ as follows. 
If $p=0$, we set $\rho_Z$ to be the identity map.
Assuming that $p>0$ and given $\eta\in \Pa^+_n$, let $B(\eta) = \{ b_1<\dots <b_N\}$ and 
$\bw (Z, \eta) = (w_1,\dots,w_h)$. Set $w_0=0$ and 
$w_{h+1}=b_N$ (with $b_N:=0$ if $B(\eta)=\varnothing$). Then we define
$\rho_Z (\eta)\in P^+_n$ by the conditions that $\rho_Z(\eta)[0,b_1]=\eta[0,b_1]$, $\rho_Z(\eta)[b_N,n]= \eta[b_N,n]$ 
and for $1 \leq k <N$, 
\[
\rho_Z (\eta) [b_k,b_{k+1}] = 
\begin{dcases*}
s_m \cdot \eta[b_k,b_{k+1}] & if $\pi(b_k)=\pi(b_{k+1})\in H_m$ for some $m\in \Z_{>0}$ \\
& \qquad and $[b_k,b_{k+1}] \subseteq [w_i,w_{i+1}]$ for some odd $i\in [1,h]$; \\
\eta[b_k,b_{k+1}] & otherwise.
\end{dcases*}
\]
In other words, $\rho_Z (\eta)$ is obtained from $\eta$ by reflecting each arc that is a segment of $\eta[w_i,w_{i+1}]$ for {\em odd} $i$ with respect to the wall where the arc begins and ends. 

We consider the corresponding map 
\[
\rho_Z\colon \Std_{\le 2} (n) \longrightarrow \Std_{\le 2} (n)
\]
 on standard tableaux, defined by the condition that 
$\rho_Z (\pi_{\tt}) = \pi_{\rho_Z (\tt)}$ for all $\tt\in \Std_{\le 2} (n)$. Since $\rho_Z (\eta)(n) = \eta(n)$ for all $\eta\in P^+_n$, the map $\rho_Z$ leaves $\Std(\mu)$ invariant for each $\mu\in \Par_{\le 2}(n)$. 
Moreover, $\rho_Z$ leaves $\DStd_{e,p}(\mu)$ invariant because $B(\rho_Z (\eta))= B(\eta)$ and $\rho_Z (\eta) (b) = \eta(b)$ 
for all $\eta\in P^+_n$ and $b\in B(\eta)$. 

Let $\la \in \Par_{\le 2} (n)$ and $\mu\in \RPar_{e,\le 2} (n)$ be such that 
$\Std_{e,p,\mu} (\la) \ne \varnothing$. 
Define 
\begin{equation}\label{Ereplamu}
r_{e,p,\la,\mu}:= r_e (\tt) \in \{0,1\}
\end{equation}
for any $\tt\in \Std_{e,p,\mu} (\la)$. By Lemma~\ref{DecNo}, 
the right-hand side does not depend on the choice of $\tt$ and 
\begin{equation}\label{Ereplamu2}
q^{r_{e,p,\la,\mu}} = [S^\la_F : D^\mu_F]_q.
\end{equation}

Let $Z$ be the regularisation set of any $\tt \in \Std_{e,p,\mu}(\la)$. Note that $Z$ does not depend on $\tt$ by Lemma~\ref{existence}(ii). Define the map 
\begin{equation}\label{Eregdashed}
\reg'_{e,p,\la,\mu} \colon \Std_{e,p,\mu} (\la) \longrightarrow \DStd_{e,p} (\mu), \; \tt \longmapsto \rho_Z (\reg_{e,p} (\tt)). 
\end{equation}
If $p=0$, then 
$\reg'_{e,p,\la,\mu}$ is simply the restriction of $\reg_{e,0}=\reg_e$ to $\Std_{e,p,\mu} (\la)$.

\begin{lemma} \label{deg_prprty}  Let $\lam\in \Par_{\le 2} (n)$ and $\mu\in \RPar_{\le 2} (n)$. 
If $\mtt t\in\Std_{e,p,\mu}(\lam)$, then 
\[
\deg_e \!\big(\reg'_{e,p,\la,\mu} (\tt) \big)=\deg_e(\mtt t)-r_{e,p,\la,\mu}. 
\]
\end{lemma}

\begin{proof} 
When $p=0$, the result follows from Lemma~\ref{LSubtr}, so we assume that $p>0$. 
Let $\mtt s=\reg_{e,p}(\mtt t)$
 and $Z$ be the regularisation set of $\tt$,
 so that $\reg'_{e,p,\la,\mu} (\tt) = \rho_Z (\st)$. 
Write $\bw(Z,\pi_{\tt}) = (w_1,\dots,w_h)$. 
Let $w_{h+1}\in [0,n]$ be maximal such that $\pi_{\tt} (w_{h+1}) \in H$, with $w_{h+1}=0$ if no such number exists. 
By Lemma~\ref{LDegArcs1}(ii), for any $0\le i\le h$ we have 
\[
\deg_e(\pi_{\rho_Z (\mtt s)}[w_i,w_{i+1}])=
(-1)^i \deg_e(\pi_\mtt s[w_i,w_{i+1}]). 
\]
Hence, 
\begin{align*}
\deg_e \!\big(\pi_{\rho_Z (\mtt s)}\big) &=
\sum_{i=0}^h (-1)^i \deg_e (\pi_{\st} [w_i, w_{i+1}]) 
=\sum_{i=0}^h\deg_e(\pi_\mtt t[w_i,w_{i+1}])
=\deg_e(\mtt t)- r_e (\tt) \\
&= \deg_e (\tt) - r_{e,p,\la,\mu},
\end{align*}
where the second equality is due to Lemma~\ref{LDegReg} and 
we have used Lemma~\ref{LDegSum} for the first and third equalities.
\end{proof}

Recall the definitions~\eqref{Ech1} and~\eqref{Ech2} of $q$-characters.

\begin{thm}\label{Treplamu}
Let $\la\in \Par_{\le 2} (n)$ and $\mu\in \RPar_{\le 2} (n)$ with $\Std_{e,p,\mu} (\la) \ne \varnothing$. 
\begin{enumerate}[(i)]
\item The map $\reg'_{e,p,\la,\mu} \colon \Std_{e,p,\mu} (\la) \to \DStd_{e,p} (\mu)$ is a bijection. 
\item For all $\tt\in \Std_{e,p,\mu} (\la)$, we have 
$\bi^{\reg'_{e,p,\la,\mu} (\tt)} = \bi^\tt$.
\item  We have 
$\fch \Std_{e,p,\mu} (\la) = q^{r_{e,p,\la,\mu}} \fch \DStd_{e,p} (\mu)$. 
\end{enumerate}
\end{thm}

\begin{proof}
By Lemma~\ref{existence}(i), 
$\reg_{e,p}$ restricts to a bijection from  $\Std_{e,p,\mu} (\la)$ onto 
$\DStd_{e,p} (\mu)$. Clearly, $\rho_Z$ restricts to an involution of $\DStd_{e,p} (\mu)$ for any set $Z$ as in~\eqref{Eregdashed}, 
and (i) follows. Part (ii) follows from Lemmas~\ref{LResSeq} and~\ref{LReflResSeq}.
Part (iii) follows immediately from (i), (ii) and Lemma~\ref{deg_prprty}. 
\end{proof}

\begin{Examp} 
Let $e=3$, $p=2$, $n=29$ and $\la=(2^{11},1^7)$. Suppose that 
$\tt \in \Std(\la)$ is determined from 
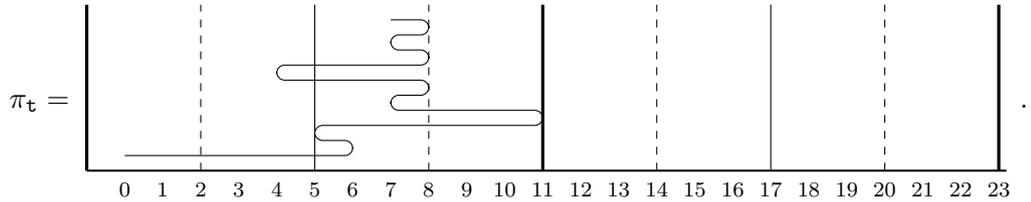
\[
\hspace{17mm}
\pi_\tt = \hspace{1mm}
\begin{tikzpicture}[baseline={([yshift=-0.25ex]current bounding box.center)},scale=0.5]
\draw[thick] (-1,0)--(23.2,0);
\draw[very thick] (-1,0)--(-1,4.4);
\foreach \k in {0,1,2,...,23} {\node at (\k,-0.5){\scriptsize$\k$};};
\foreach \k in {2,8,14,20} {\draw[dashed](\k,0)--(\k,4.4);};
\foreach \k in {5,17} {\draw(\k,0)--(\k,4.4);};
\foreach \k in {11,23} {\draw[very thick] (\k,0)--(\k,4.4);};
\draw[rounded corners=3pt]
(0,0.4)--(6,0.4)--(6,0.8)--(5,0.8)--(5,1.2)--(11,1.2)--(11,1.6)--(7,1.6)--
(7,2)--(8,2)--(8,2.4)--(4,2.4)--(4,2.8)--(8,2.8)--(8,3.2)--(7,3.2)--(7,3.6)--(8,3.6)--(8,4)--(7,4);
\end{tikzpicture}.
\]
Then the regularisation set of $\tt$ is $Z=\{2,1,0\}$. Further, 
 $\bw (Z, \pi_\tt) = (w_1,w_2,w_3)$ with $w_1 = 13$, 
$w_2= 23$ and $w_3=28$, 
 and 
$B(\pi_\tt) = \{ 2, 5, 7, 10, 13, 16, 18, 21, 23, 26, 28\}$.
Thus, 
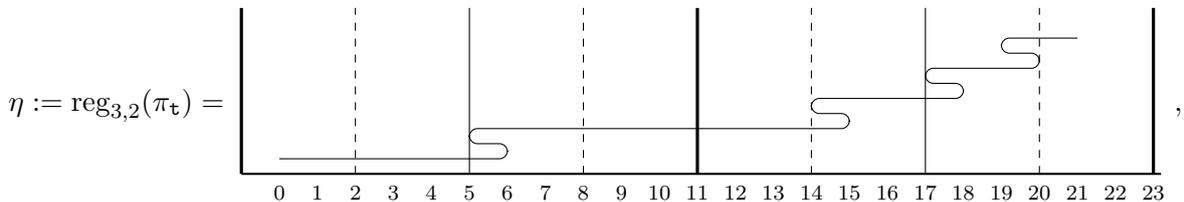
\[
\hspace{-3mm}
\eta:=\reg_{3,2} (\pi_\tt) =  \hspace{1mm}
\begin{tikzpicture}[baseline={([yshift=-0.25ex]current bounding box.center)},scale=0.5]
\draw[thick] (-1,0)--(23.2,0);
\draw[very thick] (-1,0)--(-1,4.4);
\foreach \k in {0,1,2,...,23} {\node at (\k,-0.5){\scriptsize$\k$};};
\foreach \k in {2,8,14,20} {\draw[dashed](\k,0)--(\k,4.4);};
\foreach \k in {5,17} {\draw(\k,0)--(\k,4.4);};
\foreach \k in {11,23} {\draw[very thick] (\k,0)--(\k,4.4);};
\draw[rounded corners=3pt]
(0,0.4)--(6,0.4)--(6,0.8)--(5,0.8)--(5,1.2)--(15,1.2)--(15,1.6)--(14,1.6)--
(14,2)--(18,2)--(18,2.4)--(17,2.4)--(17,2.8)--(20,2.8)--(20,3.2)--(19,3.2)--(19,3.6)--(21,3.6);
\end{tikzpicture}, 
\]
so that $\tt\in \Std_{3,2,\mu} (\la)$ where $\mu= (2^4, 1^{21})$.
(We use the same convention on thickness of walls as in Example~\ref{Ex9tab}.)
Further, 
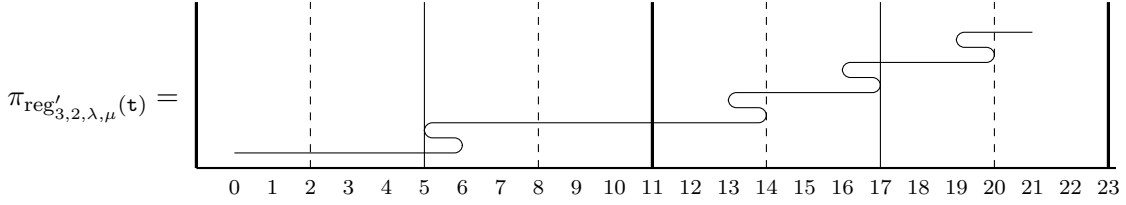
\[
\hspace{2mm}
\pi_{\reg'_{3,2,\la,\mu} (\tt)} =  \hspace{1mm}
\begin{tikzpicture}[baseline={([yshift=-0.25ex]current bounding box.center)},scale=0.5]
\draw[thick] (-1,0)--(23.2,0);
\draw[very thick] (-1,0)--(-1,4.4);
\foreach \k in {0,1,2,...,23} {\node at (\k,-0.5){\scriptsize$\k$};};
\foreach \k in {2,8,14,20} {\draw[dashed](\k,0)--(\k,4.4);};
\foreach \k in {5,17} {\draw(\k,0)--(\k,4.4);};
\foreach \k in {11,23} {\draw[very thick] (\k,0)--(\k,4.4);};
\draw[rounded corners=3pt]
(0,0.4)--(6,0.4)--(6,0.8)--(5,0.8)--(5,1.2)--(14,1.2)--(14,1.6)--(13,1.6)--
(13,2)--(17,2)--(17,2.4)--(16,2.4)--(16,2.8)--(20,2.8)--(20,3.2)--(19,3.2)--(19,3.6)--(21,3.6);
\end{tikzpicture}. 
\]
Here, $\pi_{\reg'_{3,2,\la,\mu} (\tt)}$ is obtained from $\eta$ by 
reflecting the arc $\eta[16,18]$ with respect to the wall $H_5 = \{14\}$ and the arc $\eta[21,23]$ with respect to the wall 
$H_6 = \{18\}$. 
We have $\deg_3(\tt)=3$, $\deg_3 (\eta)=-2$  
and $\deg_3\!\big(\!\reg'_{3,2,\la,\mu}(\tt)\big)=2=
\deg_3(\tt)-r_{3,2,\la,\mu}$, cf.~Lemma~\ref{deg_prprty}. 
Also, this example illustrates Lemma~\ref{LDegReg}, as 
\begin{align*}
&\deg(\eta[0,w_1]) = -1 = \deg(\pi_\tt [0,w_1]),  \quad 
\deg(\eta[w_1,w_2]) = -2 = - \deg(\pi_t [w_1,w_2]) 
\quad \text{and}\\
&\deg(\eta[w_2,w_3]) = 1= \deg(\pi_\tt [w_2,w_3]).
\end{align*}
\end{Examp}

\begin{proof}[Proof of Theorem~\ref{TMainChar}]
We use induction on $\la$ with respect to the dominance order 
$\domby$, recalling that $[S^\la_F : D^\mu_F]=0$ unless 
$\mu\domby \la$. 
Observe that if $\mu\in \Par(n)$ and $\mu\domby \la$, then $\mu\in \Par_{\le 2} (n)$. 
Then: 
\begin{align*}
\fch D^\la_F &= \fch S^\la_F - 
\sum_{\mu\triangleleft \la}[S^\la_F :D^\mu_F]_q  \fch D^\mu_F \\
&= \fch \Std(\la) - \sum_{\mu\triangleleft \la} q^{r_{e,p,\la,\mu}} 
\fch \DStd_{e,p} (\mu)  \\
&= \fch \Std(\la) - \sum_{\mu\triangleleft \la} \fch \DStd_{e,p,\mu} (\la) \\
&= \fch \DStd_{e,p} (\la),
\end{align*}
 where the first and last equalities follow from the definitions, 
 the second equality is due to \eqref{ESpechtChar}, \eqref{Ereplamu2} and the inductive hypothesis, and the third equality holds by Theorem~\ref{Treplamu}(iii). 
 \end{proof}

\begin{Remark}\label{RInv}
Recall the involution~\eqref{EInvChar} on $\langle I^n \rangle$. 
Given $\la\in\Par_{\le 2} (n)$, the simple $R^{\La_0}_{n,F}$-module $D^\la_F$ is self-dual by~\cite[Theorems 4.11, 5.13, 5.18]{BK09}, which implies that $\overline{\fch D^\la_F} = \fch D^\la_F$.
By Theorem~\ref{TMainChar} (and Theorem~\ref{TMain1}(ii)), it follows that 
\begin{equation}\label{EChInvIdent}
  \overline{\fch \DStd_{e,p} (\la)} = \fch \DStd_{e,p} (\la).
\end{equation}
The last identity can also be proved combinatorially via the
 involution $\iota \colon \Pa^+_n \to \Pa^+_n$ defined as follows. 
Given $\pi\in \Pa^+_n$, let $B(\pi) = \{ b_1<\dots <b_N\}$ and define 
$\iota(\pi) \in \Pa^+_n$ by the conditions that 
\begin{align*}
\iota(\pi) [0, b_1] &= \pi[0,b_1], \quad \iota(\pi)[b_N,n] = \pi[b_N,n],
\\ 
\iota(\pi)[b_i,b_{i+1}] 
&=
\begin{cases}
s_m \cdot \pi[b_i,b_{i+1}] & \text{if } \pi[b_i,b_{i+1}] \text{ is an arc, with } \pi(b_i) = \pi(b_{i+1}) \in H_m; \\
\pi[b_i,b_{i+1}] & \text{otherwise}
\end{cases}
\end{align*} 
whenever $1\le i<N$.
 That is, $\iota(\pi)$ is obtained from $\pi$ by reflecting all arcs. 
 As usual, this yields an involution $\iota\colon \Std_{\le 2} (n) \to \Std_{\le 2} (n)$ determined by $\pi_{\iota(\tt)} = \iota(\pi_\tt)$ for all $\tt \in \Std_{\le 2} (n)$. Then $\iota$ restricts to an involution on $\DStd_{e,p} (\la)$. Moreover, by Lemmas~\ref{LReflResSeq} and~\ref{LDegArcs1}(i), 
for all $\tt\in \DStd_e (\la)$ we have 
 \[
  \bi^{\iota(\tt)} = \bi^{\tt}, \quad \deg(\iota(\tt))=-\deg(\tt),
 \]
which implies~\eqref{EChInvIdent}. 
Combining~\eqref{EChInvIdent} with Theorem~\ref{Treplamu}(iii), we obtain the inequality 
\begin{equation}\label{EIneq}
 \sum_{\mathclap{\tt\in \Std(\la, \bi)}} \; \deg(\tt) \ge 0
\end{equation}
for any fixed $\bi \in I^n$. It is proved 
in~\cite[Corollary 3.16]{HM15} via determinants of Gram matrices for Specht modules 
that this inequality holds for an {\em arbitrary} (multi)partition $\la$, but in general no combinatorial proof is known (even when $\la$ is an $e$-core, in which case~\eqref{EIneq} becomes an equality). 
\end{Remark}

%%%%%%%%%% HOMOMORPHISMS BETWEEN SPECHT MODULES %%%%%%%%%%

\section{Homomorphisms between 2-column Specht modules}\label{SHom}

Let $\O$ be an arbitrary (unital) commutative ring and $n\in \Z_{\ge 0}$. 
We recall a row removal result 
from~\cite{fs16} for homomorphisms between Specht modules, applied in the special case of the algebra $R^{\La_0}_n:=R^{\La_0}_{n,\O}$.

Let $\la,\mu \in \Par(n)$. Recall from Proposition~\ref{Pbasis} that 
$S^\la_{\O}$ has a basis $\{ v^{\tt} \mid \tt\in\Std (\la)\}$. 
Also, recall the tableaux $\tt^\mu$ and $\tt_\mu$ from~\S\ref{SSPartitions}. 
A standard tableau $\tt\in \Std (\la)$ is {\em $\mu$-row-dominated}
if, whenever $\tt^{-1}(a,b) = (\tt^\mu)^{-1} (c,d)$ for some $(a,b) \in \Y\la$ and $(c,d)\in \Y\mu$, we have $a\leq c$.
%for each $r\in [1,n]$, the number of the row containing $r$ in $\tt$ is less than or equal to the number of the row containing $r$ in $\tt^\mu$, that is, if 
Denote the set of $\mu$-row-dominated elements 
of $\Std(\la)$ by $\Std^\mu(\la)$. 
A homomorphism 
$\phi\in \Hom_{R^{\La_0}_n} (S^\mu_\O, S^\la_\O)$ is said to be \emph{(row-)dominated} if 
$\phi(v^\mu) \in \langle v^\tt \mid \tt \in \Std^\mu (\la) \rangle_\O$.
We denote by 
$\DHom_{R^{\La_0}_n} (S^\mu_\O, S^\la_\O)$ 
the set of dominated homomorphisms 
lying in $\Hom_{R^{\La_0}_n} (S^\mu_\O, S^\la_\O)$. 

If $\la=(\la_1,\dots,\la_l)$ is a partition and $\tt\in \Std(\la)$, let 
$\bar \la:=(\la_2,\dots,\la_l)$ and let $\bar \tt\in \Std(\bar \la)$ be the tableau obtained from $\tt$ by removing the first row and decreasing all entries by $\la_1$.  

\begin{thmc}{fs16}{Theorem~4.1}\label{rowrem}
Let $\la,\mu\in \Par(n)$ be such that $\la_1=\mu_1$. 
There is an isomorphism 
\[
\DHom_{R^{\La_0}_n} \!\big(S^\mu_\O,S^\la_\O\big) \iso 
\DHom_{R^{\La_0}_{n-\la_1}} \!\big(S^{\bar{\mu}}_\O,S^{\bar\la}_\O\big), \; \phi \longmapsto \bar \phi
\]
of graded $\O$-modules. 
Moreover, if $\varphi\in \Hom_{R^{\La_0}_n}(S^\mu_\O,S^\la_\O)$ 
is given by
\[
\varphi(v^\mu) = \sum_{\mathclap{\tt \in \Std(\la)}} \; a_\tt v^\tt 
\]
for some coefficients $a_\tt \in \O$,
then 
\[
\bar\varphi(v^{\bar\mu}) = \sum_{\mathclap{\tt \in \Std(\la)}}\;a_{\tt} v^{\bar{\tt}}.
\]
\end{thmc}

\begin{Remark} 
\begin{enumerate}[(i)]
\item  This result is stated in~\cite{fs16} in terms of column Specht modules when $\O$ is a field. However, the proof works for an arbitrary commutative ring $\O$ and can be translated to the present set-up by transposing all partitions and tableaux. The second assertion of Theorem~\ref{rowrem} is clear from the proof of~\cite[Theorem 4.1]{fs16}.  
\item By~\cite[Theorem 3.6]{fs16}, whenever $e>2$, we have 
$\DHom_{R^{\La_0}_n}(S^\mu_\O,D^\la_\O) = 
\Hom_{R^{\La_0}_n} (S^\mu_\O, S^\la_\O)$ for all $\la,\mu\in \Par(n)$. 
\end{enumerate}
\end{Remark}

Our next goal is to apply Theorem~\ref{rowrem} to 2-column partitions. 

\begin{lemma}\label{LLength}
Let $\la\in \Par_{\le 2}(n)$. Suppose that $\tt,\st\in \Std(\la)$ satisfy  $\st \domby \tt$, and let $w\in \Si_n$ be determined from 
$\st=w\tt$. Then $\ell(d(\st)) = \ell(d(\tt))+\ell(w)$. 
\end{lemma}

\begin{proof}
We may assume that $\st \ne \tt$. 
Let $r\in [1,n]$ be smallest such that $\st(r)\ne \tt(r)$. Then
$\st(r) = (a+1,1)$ and $\tt(r)=(b+1,2)$ for some $a,b\in \Z_{\ge 0}$. Let $v = \tt^{-1}(a+1,1)$ and $\ut=s_{v-1}\tt$. 
Then $\ell(d(\ut))=\ell(d(\tt))+1$. We claim that 
$\st\domby \ut$. Indeed, for $j\in [1,n] \setminus\{v-1\}$, 
we have 
$\Shape(\ut\da_j) = 
\Shape(\tt\da_j) \dom \Shape(\st\da_j)$.
On the other hand, the second column of $\Shape(\st\da_{v-1})$ has size at most $b+v-1-r$, 
which is equal to the size of the second column of 
$\Shape(\ut\da_{v-1})$. This proves the claim. Arguing by induction, we may assume that the lemma holds with $\tt$ replaced by $\ut$, whence it follows that 
\[
\ell(d(\st)) = \ell(d(\ut))+ \ell(ws_{v-1}) = \ell(d(\ut))+ \ell(w)-1 = \ell(d(\tt))+ \ell(w). \qedhere
\]
\end{proof}

\begin{lemma}\label{LIndep}
Let $\la\in \Par_{\le 2} (n)$ and $\tt\in \Std(\la)$.
\begin{enumerate}[(i)]
\item 
The element $v^\tt \in S^\la_\O$ does not depend on the choice of a reduced expression for $d(\tt) \in \mathfrak S_n$. 
\item Let $\st \in \Std(\la)$ be such that $\st \domby \tt$. If 
$w\in \mathfrak S_n$ satisfies $w\tt = \st$ then $\psi_w v^{\tt} = v^{\st}$. 
\end{enumerate}
\end{lemma}

\begin{proof}
By~\eqref{rel:commpsi}, in order to prove (i) it suffices to show that $w:=d(\tt)$ is fully commutative, i.e. that every reduced expression for $w$ can be obtained from any other by using only relations of the form $s_i s_j = s_j s_i$ for $i,j\in [1,n-1]$ with $|i-j|> 1$. 
By~\cite[Theorem~2.1]{BJS93}, this is equivalent to showing that for no triple $1\le i<j<k\le n$ it is the case that 
$wi>wj>wk$. This last statement is clear since two of $i,j,k$ must lie in the same column of $\tt$. 
Part (ii) follows from (i) because a reduced expression for $d(\st)$ can be obtained by concatenating reduced expressions for $d(\tt)$ and $w$, by Lemma~\ref{LLength}. 
\end{proof}

We note that the statements of Lemmas~\ref{LLength} and~\ref{LIndep} are generally false for partitions with arbitrarily many columns. An analogue of Lemma~\ref{LIndep}(i) for 2-row partitions is~\cite[Lemma 3.17]{KMR12}. 

In the rest of the section, we assume that $\O=\Z$ and write 
$R^{\La_0}_n:=R^{\La_0}_{n,\Z}$. The results proved over $\Z$ below may be seen to hold over an arbitrary commutative ring by extending scalars.

\begin{lemma}\label{trivhoms}
Suppose that $n \equiv j-1 \bmod  e$ with $1 \leq j < e$ and $n>j-1$. 
Let $\mu = (1^n)$ and $\la= (2^j, 1^{n-2j})$. 
Then there is a dominated $R_n^{\Lambda_0}$-homomorphism 
$\phi\colon S^\mu_\Z \rightarrow S^\la_\Z$ 
given by $v^\mu \mapsto v^{\mtt {t}_\la}$. Moreover, $\phi$ is homogeneous of degree $1$. 
\end{lemma}

\begin{proof}
By Theorem~\ref{p0_thm} and Lemma~\ref{altform}(ii), 
the $R^{\La_0}_{n,\Q}$-module $S^\la_\Q$ has exactly two graded composition factors, namely $D^\la_\Q$ and 
$D^\mu_\Q\langle 1 \rangle$, so there is a non-zero homomorphism of degree $1$ from 
$S^\mu_{\Q}$ to $S^\la_{\Q}$.
 It is easy to see that $\tt_\la$ is the only standard $\la$-tableau with residue sequence $\bi^{\mu}$, so a scalar multiple of this homomorphism sends $v^\mu$ to $v^{\tt_{\la}}$. 
In view of Corollary~\ref{CScalarExt}, the lemma follows.
\end{proof}

Consider a partition $\la = (2^x, 1^y)\in \Par_{\le 2} (n)$ such that 
$y \equiv -j-1 \pmod e$ and $x\ge j$ for some $1\le j<e$. 
We define $\Tt^\la_e$ to be the tableau obtained by putting $\tt^{(2^{x-j})}$ on top of the tableau $\tt_{(2^{j}, 1^y)}$, with all entries in the latter tableau increased by $2(x-j)$. 
In other words, for all $(a,b) \in \Y\la$, we set
\[
(\Tt^\la_e)^{-1} (a,b):= 
\begin{cases}
 2(a-1)+b & \text{if } a\le x-j; \\
 x-j+a & \text{if } a>x-j \text{ and } b=1; \\
 x+y+a & \text{if } a>x-j \text{ and } b=2. 
\end{cases}
\]

\begin{thm}\label{Tphilamu}
Let $\la = (2^x,1^y)\in \Par_{\le 2} (n)$. Suppose that 
$y \equiv -j-1 \pmod e$ for some 
$1\le j<e$ and that $x\ge j$. 
Let $\mu =(2^{x-j}, 1^{y+2j})$. Then there is a  dominated $R^{\La_0}_n$-module homomorphism 
\[
 \phi_{\la,\mu} \colon S^\mu_\Z \longrightarrow S^\la_\Z, \; v^{\mu}\longmapsto v^{\Tt^\la_e},
 \]
 which is homogeneous of degree $1$. 
\end{thm}

\begin{proof}
This follows from Lemma~\ref{trivhoms} and Theorem~\ref{rowrem} 
applied $x-j$ times. 
\end{proof}

\begin{Remark}
Up to scalar multiples, the only non-zero homomorphisms between Specht modules 
$S^\la_{\Q}$ and 
$S^\mu_{\Q}$ for distinct $\la,\mu\in\Par_{\le 2}(n)$ 
are obtained from those given by Theorem~\ref{Tphilamu} by extending scalars. This follows from a consideration of composition factors of these Specht modules given by Theorem~\ref{p0_thm}.
\end{Remark}

The aim of the rest of this section is to describe the kernel and image of the homomorphism $\phi_{\la,\mu}$ from Theorem~\ref{Tphilamu} in terms 
of $e$-regular tableaux. 

 Fix $\la=(2^x,1^y)$ and $j\in [1,e-1]$ satisfying the hypotheses
of Theorem~\ref{Tphilamu}. 
Let $m\in \Z_{>0}$ be such that $y=me-j-1$.
Recall the set $H$ from~\eqref{EH}. 
Given $0 \leq r \leq x-j$, set 
\begin{align*}
Q_r:= \{\tt \in \Std(\lam) & \mid 
\pi_{\tt} (2r+me-1) \in H_m \text{ and }  
\pi_\tt (a) \notin H 
 \text{ whenever } 
2r+me-1 < a \le n \}. 
\end{align*}
That is, a standard $\la$-tableau $\tt$ lies in $Q_r$ if and only if the path
$\pi_\tt$ hits the wall $H_m$ at step $2r+me-1$ and does not hit any walls at later steps. Then 
\begin{equation}\label{EunionQ}
\Std(\la) \sm \DStd_e (\la) = \bigsqcup_{r=0}^{x-j} Q_r. 
\end{equation}
If $e=2$, then $j=1$ and $Q_r = \varnothing$ for all $r<x-j$.
 If $e>2$, then each set $Q_r$ on the right-hand side of~\eqref{EunionQ} is non-empty. 

\begin{lemma}  \label{tailend}
Suppose that $0 \leq r \leq x-j$ and that $\tt,\st \in Q_r$ satisfy
$\tt \da_{2r+me-1} = \st \da_{2r+me-1}$. 
If $w\in \Sym_n$ is such that $\st = w \tt$, then $v^{\st} = \psi_w v^{\tt}$.
\end{lemma}

\begin{proof}
First, note that we can turn $\tt$ into $\st$ by a series of elementary transpositions of the form $s_a$ with $2r+me-1<a<n$.
Hence, it suffices to prove the lemma when $w=s_a=(a,a+1)$ is an elementary transposition with $a>2r+me-1$. 

 If $a+1$ lies in the first column of $\tt$ and $a$ lies in the second, then $\tt \triangleright \st$, so 
 $v^{\st} = \psi_a v^{\tt}$ by Lemma~\ref{LIndep}(ii), as required.  
So we may assume that $a$ lies in the second column of $\st$, whence it follows that $v^{\tt} = \psi_a v^{\st}$.
Let $k$ and $l$ be determined from $\st(a)=(k,2)$ and 
$\st(a+1)=(l,1)$, 
and set $\bi:= \bi^\st$.  Then 
\[
(m-1)e -1 < \pi_\st (a) =l-k-1< \pi_\st (a+1) = l-k < \pi_\tt(a) = l-k+1 <me-1.
\]
We have $i_a = 2-k+e\Z$ and $i_{a+1} = 1-l+e\Z$, so 
\[
i_a - i_{a+1} = 1+l-k + e \Z \notin \{ 0, 1, -1\}.
\]
Hence, $\psi_a v^\tt = \psi_a^2 v^\st = v^\st$ by~\eqref{rel:quad}.  
\end{proof}

\begin{defn}
Suppose that $e>2$. 
For each $r=0,\dots,x-j$, we define a special element $\Tt_r\in Q_r$
as follows: 
\begin{enumerate}[(a)]
\item $\Tt_{x-j}:= \Tt^\la_e$. 
\item Suppose that $0\le r<x-j$. Then we set $\nu^r:=(2^r, 1^{me-1})\in \Par(2r+me-1)$ and define $\Tt_r$ to be an arbitrary (fixed) element of $\Std(\la)$ satisfying the following conditions:
\begin{enumerate}[(i)]
\item
$\Tt_r \da_{2r+me-1} = \tt^{\nu^r}$;
\item 
$2r+me$ and $2r+me+1$ are in the second column of $\Tt_r$, and 
$2r+me+2$ is in the first column;
\item 
whenever $2r+me+2 <a \le n$, we have $(m-1)e-1 < \pi_{\Tt_r} (a) <me-1$
(or, equivalently, $\Tt_r \in Q_r$). 
\end{enumerate}
\end{enumerate}

Further, for $0\le r< x-j$, define 
\[
 \St_{r+1} := s_{2r+me} s_{2r+me+1} \Tt_{r} \in Q_{r+1}.
\]
\end{defn}

\begin{Examp}
Let $e=5$ and $\la=(2^5,1^6)$. Then $x-j=2$, and the paths $\pi_{\Tt_r}$ and $\pi_{\St_r}$ (for a possible choice of $\Tt_0$ and $\Tt_1$) are as follows:
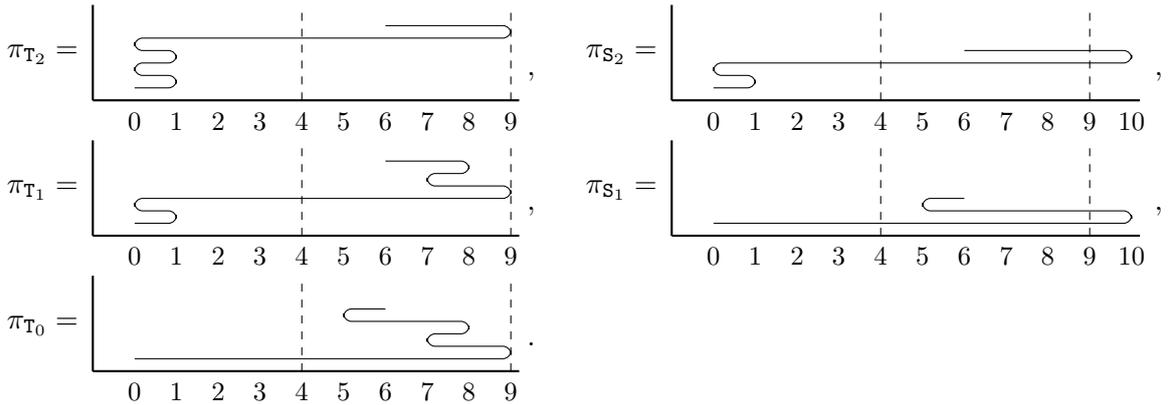
\[
\begin{array}{cc}
\begin{tikzpicture}[baseline={([yshift=-0.25ex]current bounding box.center)},scale=0.55]
\node at (-2.2,0.8) {$\pi_{\Tt_2}=$};
\draw[thick] (-1,-0.3)--(9.2,-0.3);
\draw[thick] (-1,-0.3)--(-1,2);
\foreach \k in {0,1,2,...,9} {\node at (\k,-0.8){\small$\k$};};
\foreach \k in {4,9} {\draw[dashed](\k,-0.3)--(\k,2);};
\draw[rounded corners=3pt]
(0,0)--(1,0)--(1,0.3)--(0,0.3)--(0,0.6)--(1,0.6)--(1,0.9)--(0,0.9)--(0,1.2)--(9,1.2)--(9,1.5)--(6,1.5);
\end{tikzpicture}, & \quad
\hspace{-2mm}
\begin{tikzpicture} [baseline={([yshift=-0.25ex]current bounding box.center)},scale=0.55]
\draw[thick] (-1,-0.3)--(10.2,-0.3);
\draw[thick] (-1,-0.3)--(-1,2);
\node at (-2.2,0.8) {$\pi_{\St_2}=$};
\foreach \k in {0,1,2,...,10} {\node at (\k,-0.8) {\small$\k$};};
\foreach \k in {4,9} {\draw[dashed](\k,-.3)--(\k,2);};
\draw[rounded corners=3pt]
(0,0)--(1,0)--(1,0.3)--(0,0.3)--(0,0.6)--(10,0.6)--(10,0.9)--(6,0.9);
\end{tikzpicture},
\\
\begin{tikzpicture} [baseline={([yshift=-0.25ex]current bounding box.center)},scale=0.55]
\draw[thick] (-1,-0.3)--(9.2,-0.3);
\draw[thick] (-1,-0.3)--(-1,2);
\node at (-2.2,0.8) {$\pi_{\Tt_1}=$};
\foreach \k in {0,1,2,...,9} {\node at (\k,-0.8) {\small$\k$};};
\foreach \k in {4,9} {\draw[dashed](\k,-.3)--(\k,2);};
\draw[rounded corners=3pt]
(0,0)--(1,0)--(1,0.3)--(0,0.3)--(0,0.6)--(9,0.6)--(9,0.9)--(7,0.9)--(7,1.2)--(8,1.2)--(8,1.5)--(6,1.5);
\end{tikzpicture}, & \quad
\hspace{-2mm}
\begin{tikzpicture} [baseline={([yshift=-0.25ex]current bounding box.center)},scale=0.55]
\draw[thick] (-1,-0.3)--(10.2,-0.3);
\draw[thick] (-1,-0.3)--(-1,2);
\node at (-2.2,0.8) {$\pi_{\St_1}=$};
\foreach \k in {0,1,2,...,10} {\node at (\k,-0.8) {\small$\k$};};
\foreach \k in {4,9} {\draw[dashed](\k,-.3)--(\k,2);};
\draw[rounded corners=3pt]
(0,0)--(10,0)--(10,0.3)--(5,0.3)--(5,0.6)--(6,0.6);
\end{tikzpicture},
\\
\begin{tikzpicture} [baseline={([yshift=-0.25ex]current bounding box.center)},scale=0.55]
\draw[thick] (-1,-0.3)--(9.2,-0.3);
\draw[thick] (-1,-0.3)--(-1,2);
\node at (-2.2,0.8) {$\pi_{\Tt_0}=$};
\foreach \k in {0,1,2,...,9} {\node at (\k,-0.8) {\small$\k$};};
\foreach \k in {4,9} {\draw[dashed](\k,-.3)--(\k,2);};
\draw[rounded corners=3pt]
(0,0)--(9,0)--(9,0.3)--(7,0.3)--(7,0.6)--(8,0.6)--(8,0.9)--(5,0.9)--(5,1.2)--(6,1.2);
\end{tikzpicture}. 
&
\end{array}
\] 
\end{Examp}

\begin{lemma} \label{sidemove}
Assume that $e>2$.
Suppose that $0 \leq r \leq x-j$ and $\tt\in Q_r$.  If $\tt = w \Tt_r$, where $w\in\Sym_n$, then $v^\tt=\psi_w v^{\Tt_r}$.
\end{lemma}

\begin{proof}
Since $\tt,\Tt_r \in Q_r$, we can write $w = w_1w_2$, where 
$w_1$ fixes the set $[2r+me,n]$ pointwise and $w_2$ fixes 
$[1,2r+me-1]$ pointwise. 
Set $\tt_1:= w_1 \Tt_r$.
By definition of $\Tt_r$, we  have $\Tt_r \da_{2r+me-1} = \tt^{\nu^r}$, where 
$\nu^r = (2^r, me-1)$. Hence, $\Tt_r \dom \tt_1$, so by Lemma~\ref{LIndep}(ii) we have $\psi_{w_1} v^{\Tt_r} = v^{\tt_1}$. 
Lemma~\ref{tailend} yields $\psi_{w_2} v^{\tt_1} = v^{\tt}$, whence the result follows. 
\end{proof}

\begin{lemma} \label{kill1}
If $e>2$ and $0 \leq r \leq x-j$ then $\psi_{2r+me} v^{\Tt_r} = 0$.  
\end{lemma}

\begin{proof}
Let $\bi = \bi^{\Tt_r}$ and $\bi' = s_{2r+me} \bi$, so that 
$(1-1_{\bi'}) \psi_{2r+me} v^{\Tt_r}=0$. It suffices to show that $\{ \tt\in \Std(\la) \mid \bi^\tt = \bi'\}=\varnothing$, for then $1_{\bi'} S^{\la}_{\Z}=0$. 

Suppose for contradiction that $\tt\in \Std(\la)$ has residue sequence $\bi'$. 
First, we claim that 
$\tt (a) = \Tt_r (a)$
whenever $2r+me+2\le a\le n$. Assuming the claim to be false, we choose $a$ to be maximal such that the equality fails. 
By maximality of $a$, 
\[ \Shape(\tt\da_{a}) = \Shape(\Tt_r\da_{a})=:(2^c,1^d)\]
for some $c,d\in \Z_{\ge 0}$. Since $a$ lies in different columns in $\tt$ and $\Tt_r$ and $i^{\tt}_a = i^{\Tt_r}_a$, 
the residues of the bottom nodes of the two columns of $(2^c,1^d)$ must be equal. However, since $(m-1)e-1<\pi_{\Tt_r} (a)<me-1$, we have $d\not \equiv -1 \pmod e$, from which it follows that these two residues are not equal. This contradiction proves the claim. 

Using this claim, we obtain
\[\Shape (\tt\da_{2r+me+1}) = \Shape(\Tt_r\da_{2r+me+1}) 
= (2^{r+me-1}, 1^{r+2})
=: \gamma.
\]
Hence,
\[\bi'_{2r+me+1} = \bi_{2r+me} = \res \!\big(\Tt_r (2r+me)\big) = \res( r+1, 2) = 
1-r +e\Z.\]
Therefore, $\res\big(\tt (2r+me+1)\big)=1-r+e\Z$. However, 
$\tt (2r+me+1)$ must be the bottom entry of either the first or the second column of $\gamma$, and these two entries have residues $2-r+e\Z$ and $-r+e\Z$ respectively, a contradiction.  
\end{proof}

\begin{lemma} \label{crossmove}
If $e>2$ and $0\leq r<x-j$, then 
$v^{\Tt_r} = -\psi_{2r+me+1} v^{\St_{r+1}}$.
\end{lemma}

\begin{proof} 
It is easy to see that $\St_{r+1} \triangleleft \Tt_r$. So, by Lemma~\ref{LIndep}(ii), $v^{\St_{r+1}} = \psi_{2r+me} \psi_{2r+me+1} v^{\Tt_r}$. 
Further, note that \[\res\!\big(\Tt_r (2r+me)\big)=
\res\!\big(\Tt_r (2r+me+2)\big) = \res\!\big(\Tt_r (2r+me+1)\big)+1=-r+1+e\Z.\]
Hence,\begin{align*}
-\psi_{2r+me+1} v^{\St_{r+1}} &= -\psi_{2r+me+1} \psi_{2r+me} \psi_{2r+me+1} v^{\Tt_r}\\
&= -(\psi_{2r+me}\psi_{2r+me+1}\psi_{2r+me}-1) v^{\Tt_r}\\
&= v^{\Tt_r},
\end{align*}
where we have used~\eqref{rel:braid} for the second equality and 
Lemma~\ref{kill1} for the third equality. 
\end{proof}

\begin{prop}\label{generate}
If $\tt\in\Std(\la) \sm \DStd_e (\la)$ then $v^\tt$ lies in the 
$R_{n}^{\Lambda_0}$-submodule of $S^\la_{\Z}$ 
generated by $v^{\Tt^\la_e}$. 
\end{prop}

\begin{proof}
If $e=2$, then $\Tt_e^\la$ dominates every element of $Q_{x-j}=\Std(\la) \sm \DStd_e (\la)$ 
(cf.~the proof of Lemma~\ref{sidemove}), 
so the result follows from Lemma~\ref{LIndep}(ii).

Assume that $e>2$. Let $U$ be the submodule in the statement of the proposition. By Lemma~\ref{sidemove}, it is enough to show that $v^{\Tt_r}\in U$ for all $0\le r\le x-j$. We use backward induction on $r$ and note that 
$v^{\Tt_{x-j}} = v^{\Tt^\la_e} \in U$, so we may assume that $r<x-j$
and $v^{\Tt_{r+1}} \in U$.  
By the inductive hypothesis and Lemma~\ref{sidemove} applied again, 
$\St_{r+1} \in U$. Hence, $v^{\Tt_r} \in U$ by Lemma~\ref{crossmove}.
\end{proof}

The hypotheses of the following theorem are the same as in Theorem~\ref{Tphilamu}.

\begin{thm}\label{TKerIm}
Let $\la=(2^x,1^y)\in \Par_{\le 2}(n)$.
Suppose that $y\equiv -j-1 \pmod e$ for some $1\le j<e$ and that 
$x\ge j$. Let $\mu=(2^{x-j},1^{y+2j})$ and let $\phi_{\la,\mu} \colon S^\mu_\Z \to S^\la_\Z$ be as in Theorem~\ref{Tphilamu}. 
Then 
\[ \ker(\phi_{\la,\mu})=\langle v^{\tt}\mid\tt\in\Std(\mu)\sm\DStd_e(\mu) \rangle_\Z \quad\text{and}\quad \phi_{\la,\mu}(S^\mu_\Z)=\langle v^{\st}\mid\st\in\Std(\la)\sm\DStd_e (\la)\rangle_\Z.\]
\end{thm}

\begin{proof}
We use backward induction on $y$. Note that the kernel and image of $\phi_{\la,\mu}$ are free $\Z$-modules by Proposition~\ref{Pbasis}. 
By Proposition~\ref{generate}, 
\begin{equation}\label{EImsup}
\phi_{\la,\mu} (S^\mu_\Z) \supseteq \langle v^{\st} \mid \st \in \Std(\la) \sm \DStd_e (\la) \rangle_\Z. 
\end{equation}
Hence, 
\begin{align}
\dim_\Z \ker (\phi_{\la,\mu}) &= |\Std(\mu)| - \dim_\Z (\phi_{\la,\mu} (S^\mu_\Z)) \le |\Std(\mu)| - |\Std(\la) \sm \DStd_e (\la)| \notag \\
&= |\Std(\mu)| - |\DStd_e (\mu)| = |\Std(\mu) \sm \DStd_e(\mu)|, 
\label{Edimker}
\end{align}
where the penultimate equality holds because, by Lemma~\ref{LPreimage}, the map 
$\reg_{e}$ restricts to a bijection $\Std(\la) \sm \DStd_e (\la) \iso \DStd_e (\mu)$.

We claim that 
\begin{equation}\label{EKersup}
\ker (\phi_{\la,\mu} ) \supseteq \langle v^{\tt} \mid \tt \in \Std(\mu) \sm \DStd_e (\mu) \rangle_\Z.
\end{equation}
Let $m\in\Z_{>0}$ be such that $y=me-j-1$. 
If $x<e$, then for all $\tt\in\Std(\mu)$ and $0\le a\le n$ 
 we have $\pi_\tt (a) \notin H_{m+1}$, from which it follows that $\tt\in \DStd_e(\mu)$ and~\eqref{EKersup} holds trivially, as its right-hand side is $0$. 
So, suppose that $x\ge e$ and let $\nu:=(2^{x-e}, 1^{y+2e})$, noting that $y+2e$ is the image of $y+2j$ under reflection with respect to $H_{m+1}$. By the inductive hypothesis, the right-hand side of~\eqref{EKersup} is exactly the image of $\phi_{\mu,\nu}$. Now $\phi_{\la,\mu} \phi_{\mu,\nu}=0$ because, 
by Theorem~\ref{p0_thm}, 
$S^\la_\Q$ and 
$S^\nu_\Q$ have no composition factors in common. 
Thus, the claim follows. 

Using~\eqref{Edimker} and~\eqref{EKersup}, we obtain the first equality in the theorem. Also,~\eqref{Edimker} is an exact equality, so the two sides of~\eqref{EImsup} have the same $\Z$-rank. This completes the proof
since the right-hand side of~\eqref{EImsup} is a pure $\Z$-submodule of $S^\la_\Z$. 
\end{proof}

\begin{Remark}
Let $\la \in \Par_{\le 2} (n)$. It is not difficult to show that if $\tt\in \DStd_e (\la)$ and $\st \in \Std(\la) \sm \DStd_e (\la)$ 
 then $\bi^\tt \ne \bi^\st$ (cf.~the claim in the proof of Lemma~\ref{kill1}). This leads to a more direct proof of the first equality in Theorem~\ref{TKerIm}. 
\end{Remark}

\begin{proof}[Proof of Theorem~\ref{TMain1}]
Let $\la=(2^x,1^y)$, where $y=me-1-j$ for some $m\in \Z_{>0}$ and $0\le j<e$. If $j=0$ or $x<j$, then $\DStd_e (\la) = \Std(\la)$ and parts (i) and (iii) of the theorem hold. Also, in this case
$S^\la_K = D^\la_K$ by Theorem~\ref{p0_thm} and Lemma~\ref{altform}, so (ii) holds as well.

On the other hand, 
if $j\ne 0$ and $x\ge j$, then (i) and (iii) follow from Theorems~\ref{Tphilamu} and~\ref{TKerIm}, and (ii) again follows from Theorem~\ref{p0_thm}.
\end{proof}

\begin{Remark}
Theorem~\ref{TMain1}(ii) is true with $K$ replaced by any field 
of characteristic $0$.
%$F$ satisfying the conditions stated at the beginning of Section~%\ref{SComb}. 
\end{Remark}

Recall the reflections $s_m$ from~\S\ref{SSTwo}: we have 
$s_m \cdot (me-1+j) = me-1-j$ for all $m\in \Z_{>0}$ and $j\in \Z$. 

\begin{corollary}\label{genexact}
Let $\la^1 = (2^{x_1}, 1^{y_1}) \in \Par_{\le 2} (n)$. Suppose that $y_1=me-1+j$ for some $m\in \Z_{>0}$ and $1\le j<e$, 
 and that $x_1<e-j$. 
For $k=1,\dots,m$, let $y_k = s_{m-k}\cdot \ldots \cdot s_m \cdot y_1$ and 
$\la^k= (2^{x_k}, 1^{y_k}) \in \Par_{\le 2} (n)$, where $x_k = (n-y_k)/2$. Then the following is an exact sequence of $R_n^{\Lambda_0}$-homomorphisms:
\[
0 \longrightarrow S^{\la^1}_\Z \xrightarrow{\phi_{\la^2,\la^1}}
S^{\la^2}_\Z \xrightarrow{\phi_{\la^3,\la^2}} S^{\la^3}_{\Z}  \longrightarrow \cdots
\longrightarrow
S^{\la^m}_\Z. 
\]
\end{corollary}

\begin{proof}
The hypothesis ensures that $\DStd_e (\la^1) = \Std(\la^1)$, whence the result follows from Theorems~\ref{Tphilamu} and~\ref{TKerIm}. 
\end{proof}

\section{Partitions with more than two columns}\label{SCounter}

In this section, we outline a natural approach to extending the definition of $\DStd_e (\la)$ to the case when a partition $\la$ has more than two 
 columns and give an example showing that this approach does not always work. For simplicity, we consider algebras over $\Q$ only, though all statements below are true with $\Q$ replaced by any field of characteristic $0$.  

Fix $n\in \Z_{\ge 0}$ and $\la\in\Par(n)$. 
Given $\tt\in \Std(\la)$ and $u\in S^\la_\Q$, we say that $u$ is a {\em $\tt$-element} if 
\[
 u = v^\tt + \; \sum_{\mathclap{\substack{\st \in \Std(\la)\\ \st \triangleright \tt}}} \; a_{\st} v^\st
\]
for some coefficients $a_{\st} \in \Q$. 
It is important to note that, whereas the elements $v^\tt$ depend on certain choices of reduced expressions, the set of all $\tt$-elements on $S^\la_\Q$ does not depend on such choices. This is the case because if one changes the reduced expression for $d(\tt)$, causing $v^\tt$ to be replaced by $v^\tt_1$, then 
$v^\tt_1 = v^\tt +  \sum_{\st \in \Std(\la), \, \st \triangleright \tt} a_{\st} v^\st$ for some coefficients $a_{\st} \in \Q$ 
(by~\cite[Proposition 4.7]{BKW11}). 

Recall from~\S\ref{SSHecke}
the bilinear form $\langle \cdot,\cdot\rangle$ on $S^\la_\Q$ and its radical $\rad S^\la_\Q$.  Note that, by the properties of the form stated after~\eqref{Eform}, for any $\tt,\st\in \Std(\la)$, we have 
$\langle v^\tt, v^\st\rangle=0$ unless $\deg(\tt) + \deg(\st)=0$ and 
$\bi^\tt = \bi^\st$. These facts are used repeatedly in the sequel. 
Define 
\[
 \IStd_e (\la):= \{ \tt \in \Std(\la) \mid \rad S^\la_\Q \text{ contains at least one } \tt\text{-element} \}. 
\]
For each $\tt\in \IStd_e(\la)$, choose a $\tt$-element $w^\tt \in \rad S^\la_\Q$.
 The set $\{ w^\tt \mid \tt \in \IStd_e (\la) \}$ is always linearly independent over $\Q$. 
In the sequel, we call the partition $\la$ {\em $e$-agreeable} if 
$\{ w^\tt \mid \tt \in \IStd_e (\la) \}$ is a basis of 
$\rad S^\la_\Q$ or, equivalently, if $|\IStd_e (\la)|=\dim (\rad S^\la_\Q)$.
Whether or not $\{ w^\tt \mid \tt \in \IStd_e (\la) \}$ spans $\rad S^\la_\Q$ does not depend on the choice of the elements $w^\tt$. 
Note that for each $\tt\in\Std(\la)$ one can choose $w^\tt$ in such a way that $w^\tt$ is homogeneous of degree $\deg_e(\tt)$ and 
$1_{\bi^\tt} w^\tt = w^\tt$.

Whenever $\la$ is $e$-agreeable, it is reasonable to define $\DStd_e (\la)$ as the complement $\Std(\la) \setminus \IStd_e (\la)$. Then 
$\fch \DStd_e (\la) = \fch D^\la_\Q$ thanks to the last observation in the previous paragraph. 
Theorem~\ref{TMain1} shows that, if $\la\in \Par_{\le 2}(n)$, then $\la$ is $e$-agreeable and the definition of $\DStd_e (\la)$ just given agrees with the combinatorial one in~\S\ref{SSTwo}. 

\begin{Remark}
If we replaced the dominance order by an arbitrary {\em total} order on $\Std(\la)$ in the definition of a $\tt$-element, then 
$\{ w^\tt \mid \tt \in \IStd_e (\la) \}$ would automatically be a basis of $\rad S^\la_\Q$ by elementary linear algebra. This construction is discussed in~\cite[\S 3.3]{HM15}. 
\end{Remark}

Hu and Mathas~\cite[Section 6]{HM15} construct a distinguished homogeneous cellular basis 
\[
 \bigsqcup_{\mathclap{\mu\in\Par(n)}} \; \{ B_{\st\tt} \mid \st,\tt\in \Std(\mu) \} 
\]
of $R^{\La_0}_{n,\Q}$, which does not depend on any choices of reduced expressions. This new cellular structure yields a basis 
$\{ B_\tt \mid \tt\in \Std(\la) \}$ of $S^\la_\Q$. It follows from~\cite[Proposition 6.7]{HM15} that $B_{\tt}$ 
is a $\tt$-element for each $\tt\in \Std(\la)$.

A conjecture of Mathas~\cite[Conjecture 4.4.1]{M15} implies, in particular, that for every $\la\in \Par(n)$ there is a subset $\mathscr T_\la$ of $\Std(\la)$ such that 
$\{ B_{\tt} \mid \tt\in \mathscr T_\la\}$ is a basis of $\rad S^\la_\Q$. 
Since each $B_{\tt}$ is a $\tt$-element, this in turn implies (via elementary linear algebra) that every partition $\la$ is $e$-agreeable. 
However, Example~\ref{ExCounter} below shows that (for $e=3$) not all partitions are 
$e$-agreeable.
Hence, there is a counterexample to~\cite[Conjecture 4.4.1]{M15}. 

The following example was discovered using a GAP~\cite{GAP4} program for calculating Gram matrices of Specht modules. The program uses results from~\cite{HM15}, especially a certain seminormal basis, and is available on the second author's website.\footnote{\href{http://web.mat.bham.ac.uk/A.Evseev/publications.html}{http://web.mat.bham.ac.uk/A.Evseev/publications.html}}
Below we give a self-contained and computer-independent verification. 

\begin{Examp}\label{ExCounter} Let $e=3$, $n=8$ and $\la= (4,3,1)$. 
Consider the tuple $\bi = (0,1,2,2,0,1,0,1)\in I^8$. 
The set $\Std(\la,\bi)$ has exactly two elements of degree $2$, namely 
\[
\tt_1 = 
\begin{array}{c}
\begin{ytableau} 1 & 2 & 3 & 5 \\ 4& 7& 8 \\ 6 \end{ytableau}
\end{array}
\quad \text{and} \quad
\tt_2 = 
\begin{array}{c}
\begin{ytableau} 
1 & 2 & 3 & 7 \\ 4 & 5 & 6 \\ 8
\end{ytableau}
\end{array}. 
\]
Further, $\Std(\la,\bi)$ has exactly one element of degree $-2$, namely 
\[
\st= 
\begin{array}{c}
\begin{ytableau} 1& 2 & 4 & 7 \\ 3 & 5 & 8 \\ 6 \end{ytableau}
\end{array}. 
\]
Note that $y^\la= y_3 y_7$. We use the reduced expressions 
\begin{equation*}
d(\tt_1) = s_6 s_7 s_4, \qquad
d(\tt_2) = s_6 s_5 s_4, \qquad
d(\st) = s_6 s_7 s_5 s_3 s_4. 
\end{equation*}
 Using~\eqref{Eform}, we compute 
(cf.~\cite[Example 3.7.9]{M15}):
\begin{align*}
\langle v^{\tt_1}, v^{\st} \rangle v^{\tt^\la} &= 
1_{\bi^\la} y_3 y_7 \psi_4\psi_7 \psi_6^2 \psi_7 \psi_5\psi_3\psi_4  v^{\tt_\la}\\
&= y_3 y_7 \psi_4\psi_7 (y_6-y_7) \psi_7 \psi_5\psi_3\psi_4 
v^{\tt_\la} \\
&= -y_3 y_7 \psi_4 \psi_7 y_7 \psi_7 \psi_5 \psi_3 \psi_4 v^{\tt^\la}\\
&= y_3 y_7 \psi_4 \psi_7 \psi_5 \psi_3 \psi_4 v^{\tt_\la}\\
&= - y_3 \psi_4 \psi_5 \psi_3 \psi_4 v^{\tt^\la} \\
&= \psi_4 \psi_5 \psi_4 v^{\tt^\la} \\
&= (\psi_5 \psi_4 \psi_5 +1) v^{\tt^\la} \\
&= v^{\tt^\la}, 
\end{align*}
where we have repeatedly used the relations~\eqref{rel:Sp1}--\eqref{rel:Sp3}, and moreover we have used~\eqref{rel:quad} for the second equality,~\eqref{rel:ypsi4} and~\eqref{rel:quad} for the third equality,
~\eqref{rel:ypsi1}--\eqref{rel:ypsi4} for the fourth, fifth and sixth equalities and~\eqref{rel:braid} for the seventh equality. 
Further, 
\begin{align*}
\langle v^{\tt_2}, v^\st \rangle v^{\tt^\la} &=  
1_{\bi^\la} y_3 y_7 \psi_4 \psi_5 \psi_6^2 \psi_7 \psi_5\psi_3 \psi_4
v^{\tt^\la} \\
&= 
 y_3 y_7 \psi_4 \psi_5 (y_6-y_7) 
\psi_7 \psi_5\psi_3 \psi_4 v^{\tt^\la}
\\
&=
 y_3 y_7 \psi_4 \psi_5 y_6 
\psi_7 \psi_5 \psi_3 \psi_4 v^{\tt^\la} \\
&= 
 y_3 y_7 \psi_4 \psi_5 \psi_7 \psi_3 \psi_4 v^{\tt^\la}  \\
&= v^{\tt^\la},
\end{align*}
where we have used~\eqref{rel:quad} for the second equality,~\eqref{rel:ypsi4},~\eqref{rel:commpsi} and~\eqref{rel:quad} for the third equality,~\eqref{rel:ypsi1}--\eqref{rel:ypsi4} for the fourth equality, and the final equality repeats the end of the previous computation.

Hence, $\langle v^{\tt_1}, v^\st \rangle =1 = \langle v^{\tt_2}, v^\st \rangle$. 
Therefore, the degree $2$ component of $1_{\bi} (\rad S^\la_\Q)$ is 
1-dimensional and is spanned by $v^{\tt_1} - v^{\tt_2}$.
However, since neither of $\tt_1$ and  $\tt_2$ dominates the other, neither of these two tableaux belongs to $\IStd_e (\la)$. 
Hence, $\la$ is not $e$-agreeable. 
\end{Examp}

%\bigskip
%%%%%%%%%% BIBLIOGRAPHY %%%%%%%%%%

%\nocite{*} % displays full reference list regardless of citations in the text
\bibliographystyle{amsplain}
\bibliography{MASL_bibliography1}

\providecommand{\bysame}{\leavevmode\hbox to3em{\hrulefill}\thinspace}
\providecommand{\MR}{\relax\ifhmode\unskip\space\fi MR }
\providecommand{\MRhref}[2]{%
  \href{http://www.ams.org/mathscinet-getitem?mr=#1}{#2}
}
\providecommand{\href}[2]{#2}
\begin{thebibliography}{10}

\bibitem{BJS93}
S.~C.~Billey, W.~Jockusch, and R.~P.~Stanley,
  \emph{\href{http://dx.doi.org/10.1023/A:1022419800503}{Some combinatorial
  properties of {Schur} polynomials}}, J.~Algebraic~Combin. \textbf{2} (1993),
  no.~4, 345--374.

\bibitem{bcs15}
C.~Bowman, A.~Cox, and L.~Speyer,
  \emph{\href{http://dx.doi.org/10.1093/imrn/rnw101}{A family of graded
  decomposition numbers for diagrammatic {Cherednik} algebras}}, Int.\ Math.\
  Res.\ Not.\ IMRN \textbf{2017} (2017), no.~9, 2686--2734.

\bibitem{bkisom}
J.~Brundan and A.~Kleshchev,
  \emph{\href{http://dx.doi.org/10.1007/s00222-009-0204-8}{Blocks of cyclotomic
  {Hecke} algebras and {Khovanov}-{Lauda} algebras}}, Invent.\ Math.
  \textbf{178} (2009), no.~3, 451--484.

\bibitem{BK09}
\bysame, \emph{\href{http://dx.doi.org/10.1016/j.aim.2009.06.018}{Graded
  decomposition numbers for cyclotomic {Hecke} algebras}}, Adv.\ Math.
  \textbf{222} (2009), no.~6, 1883--1942.

\bibitem{BKW11}
J.~Brundan, A.~Kleshchev, and W.~Wang,
  \emph{\href{http://dx.doi.org/10.1515/CRELLE.2011.033}{Graded {Specht}
  modules}}, J.\ Reine Angew.\ Math. \textbf{655} (2011), 61--87.

\bibitem{DJ86}
R.~Dipper and G.~James,
  \emph{\href{http://dx.doi.org/10.1112/plms/s3-52.1.20}{Representations of
  {Hecke} algebras of general linear groups}}, Proc.~London Math.~Soc.
  \textbf{52} (1986), no.~1, 20--52.

\bibitem{D98}
S.~Donkin, \emph{\href{http://dx.doi.org/10.1017/CBO9780511600708}{The
  {$q$-Schur} algebra}}, London Math.~Soc.~Lecture Note Ser.~253, Cambridge
  University Press, Cambridge, 1998.

\bibitem{EW13}
B.~Elias and G.~Williamson,
  \emph{\href{https://doi.org/10.1090/ert/481}{{Soergel} calculus}},
  Represent.~Theory \textbf{20} (2016), 295--374.

\bibitem{Erdmann}
K.~Erdmann, \emph{\href{https://doi.org/10.1007/BF02567828}{Tensor products and
  dimensions of simple modules for symmetric groups}}, Manuscripta Math.
  \textbf{88} (1995), 357--386.

\bibitem{fs16}
M.~Fayers and L.~Speyer,
  \emph{\href{https://doi.org/10.1007/s10801-016-0674-x}{Generalised column
  removal for graded homomorphisms between {Specht} modules}}, J.\ Algebraic
  Combin. \textbf{44} (2016), no.~2, 393--432.

\bibitem{GAP4}
The GAP~Group, \emph{\href{http://www.gap-system.org}{GAP -- Groups,
  Algorithms, and Programming, Version 4.8.3}}, 2016.

\bibitem{GW99}
F.~Goodman and H.~Wenzl,
  \emph{\href{https://doi.org/10.1155/S1073792899000136}{Crystal bases of
  quantum affine algebras and affine {K}azhdan-{L}usztig polynomials}},
  Internat.\ Math.\ Res.\ Notices (1999), 251--275.

\bibitem{HM10}
J.~Hu and A.~Mathas,
  \emph{\href{http://dx.doi.org/10.1016/j.aim.2010.03.002}{Graded cellular
  bases for the cyclotomic {Khovanov}--{Lauda}--{Rouquier} algebras of type\
  {$A$}}}, Adv.\ Math. \textbf{225} (2010), no.~2, 598--642.

\bibitem{HM15}
\bysame, \emph{\href{http://dx.doi.org/10.1007/s00208-015-1242-8}{Seminormal
  forms and cyclotomic quiver {Hecke} algebras of type\ {$A$}}}, Math.\ Ann.
  \textbf{364} (2016), no.~3, 1189--1254.

\bibitem{J78}
G.~D. James, \emph{\href{http://dx.doi.org/10.1007/BFb0067708}{The {Representation}
  {Theory} of the {Symmetric} {Groups}}}, Lecture Notes in Mathematics, vol.
  682, Springer, Berlin, 1978.

\bibitem{J84}
\bysame,
  \emph{\href{http://dx.doi.org/10.1017/CBO9780511661921}{Representations of
  {G}eneral {L}inear {G}roups}}, London Mathematical Society Lecture Note
  Series 94, Cambridge University Press, Cambridge, 1984.

\bibitem{K90}
V.~G. Kac,
  \emph{\href{http://dx.doi.org/10.1017/CBO9780511626234}{Infinite-dimensional
  {Lie} algebras}}, third ed., Cambridge University Press, Cambridge, 1990.

\bibitem{KL09}
M.~Khovanov and A.~D. Lauda,
  \emph{\href{http://dx.doi.org/10.1090/S1088-4165-09-00346-X}{A diagrammatic
  approach to categorification of quantum groups {I}}}, Represent.\ Theory
  \textbf{13} (2009), 309--347.

\bibitem{KMR12}
A.~Kleshchev, A.~Mathas, and A.~Ram,
  \emph{\href{http://dx.doi.org/10.1112/plms/pds019}{Universal graded {S}pecht
  modules for cyclotomic {H}ecke algebras}}, Proc.\ London Math.\ Soc.
  \textbf{105} (2012), 1245--1289.

\bibitem{Li14}
G.~Li, \emph{\href{http://dx.doi.org/10.1016/j.jalgebra.2017.02.022}{Integral
  basis theorem of cyclotomic {Khovanov}--{Lauda}--{Rouquier} algebras of type
  {$A$}}}, J.\ Algebra \textbf{482} (2017), 1--101.

\bibitem{L06}
S.~Lyle, \emph{\href{http://dx.doi.org/10.1080/00927870500542713}{Some results
  obtained by application of the {LLT} algorithm}}, Comm.\ Algebra \textbf{34}
  (2006), 1723--1752.

\bibitem{M99}
A.~Mathas, \emph{\href{http://dx.doi.org/10.1090/ulect/015}{Iwahori--{Hecke}
  {Algebras} and {Schur} {Algebras} of the {Symmetric} {Group}}}, University
  Lecture Series 94, American Mathematical Society, Providence, R.I., 1999.

\bibitem{M15}
\bysame, \emph{\href{http://dx.doi.org/10.1142/9789814651813_0005}{Cyclotomic
  quiver {Hecke} algebras of type\ {$A$}}}, Modular representation theory of
  finite and $p$-adic groups (K.~M. Tan and W.~T. Gan, eds.), Lecture Note
  Series, Institute for Mathematical Sciences, National University of
  Singapore, vol.~30, World Scientific, 2015, pp.~165--266.

\bibitem{R08}
R.~Rouquier, \emph{$2$-{Kac}--{Moody} algebras}, preprint,
  \href{http://arxiv.org/abs/0812.5023}{arXiv:0812.5023}, 2008.

\bibitem{S97}
W.~Soergel,
  \emph{\href{http://dx.doi.org/10.1090/S1088-4165-97-00021-6}{{Kazhdan--Lusztig}
  polynomials and a combinatoric for tilting modules}}, Represent.\ Theory
  (1997), 83--114.

\bibitem{VV99}
M.~Varagnolo and E.~Vasserot,
  \emph{\href{http://dx.doi.org/10.1215/S0012-7094-99-10010-X}{On the
  decomposition matrices of the quantized {S}chur algebra}}, Duke Math. J.
  (1999), 267--297.

\end{thebibliography}

\end{document}